\documentclass[a4paper,11pt]{amsart}

\usepackage{graphicx, amssymb}  
\usepackage{mathrsfs}   
\usepackage[numbers, sort&compress]{natbib}
\usepackage{url,hyperref} 
\usepackage[inner=2.5cm,outer=2.5cm,top=2.5cm,bottom=2.5cm,includehead,includefoot]{geometry}
\usepackage{float}
\hypersetup{citecolor=red, linkcolor=blue, colorlinks=true}

\newtheorem{thm}{Theorem}[section]
\newtheorem{prop} [thm]{Proposition}
\newtheorem{cor} [thm]{Corollary}
\newtheorem{lemma} [thm]{Lemma}
\newtheorem{problem}[thm]{Problem}

\theoremstyle{definition}
\newtheorem{definition}[thm]{Definition}
\newtheorem{remark}[thm]{Remark}

\renewcommand\leq{\leqslant} 
\renewcommand\geq{\geqslant}

\DeclareMathOperator{\Aut}{Aut}
\DeclareMathOperator{\diam}{diam}
\DeclareMathOperator{\PG}{PG}

\DeclareMathOperator{\HS}{HS}
\DeclareMathOperator{\McL}{McL}
\DeclareMathOperator{\M}{M}
\DeclareMathOperator{\J}{J}

\DeclareMathOperator{\GL}{GL}
\DeclareMathOperator{\PSL}{PSL}
\DeclareMathOperator{\GammaL}{\Gamma L}
\DeclareMathOperator{\PGammaL}{P\Gamma L}

\DeclareMathOperator{\PGaU}{P\Gamma U}
\DeclareMathOperator{\PSigmaU}{P\Sigma U}

\DeclareMathOperator{\PSU}{PSU}

\DeclareMathOperator{\PGaSp}{P\Gamma Sp}

\DeclareMathOperator{\PGammaO}{P\Gamma O}
\DeclareMathOperator{\PSO}{PSO}
\DeclareMathOperator{\POmega}{P\Omega}
\DeclareMathOperator{\SO}{SO}

\title{On $k$-connected-homogeneous graphs}
\author{Alice Devillers, Joanna B. Fawcett,  Cheryl E. Praeger, Jin-Xin Zhou}

\address[Devillers]{Centre for the Mathematics of Symmetry and Computation,
Department of Mathematics and Statistics,
The University of Western Australia,
Crawley, W.A. 6009, Australia}
\email{alice.devillers@uwa.edu.au}

\address[Fawcett]{Department of Mathematics,  Imperial College London, South Kensington Campus, 
London, SW7 2AZ, United Kingdom}
\email{j.fawcett@imperial.ac.uk}

\address[Praeger]{Centre for the Mathematics of Symmetry and Computation,
Department of Mathematics and Statistics,
The University of Western Australia,
Crawley, W.A. 6009, Australia}
\email{cheryl.praeger@uwa.edu.au}

\address[Zhou]{Department of Mathematics, Beijing Jiaotong University, 
Beijing, 100044, P.R. China}
\email{jxzhou@bjtu.edu.cn}

\keywords{locally finite graph,  homogeneous, k-connected-homogeneous, s-arc-transitive}

\thanks{This work forms part of the Australian Research Council Discovery Project grant DP130100106 of the first and third authors. The second author was supported by this same grant, the London Mathematical Society and  the European Union's Horizon 2020 research and innovation programme under the Marie Sk\l{}odowska-Curie grant agreement No.\ 746889”. The fourth author was supported by the National
Natural Science Foundation of China (11671030).
 We thank the reviewer for helpful comments.  
}

\subjclass[2010]{20B25, 05C75, 05E18, 05E20, 05E30}

\begin{document}
    
\maketitle

\begin{abstract}
A graph $\Gamma$ is $k$-connected-homogeneous ($k$-CH) if $k$ is a positive integer and  any isomorphism between connected induced subgraphs  of order at most $k$ extends to an automorphism of $\Gamma$, and connected-homogeneous (CH) if this property holds for all $k$. 
 Locally finite, locally connected graphs often fail to be $4$-CH  because of a combinatorial obstruction called the unique $x$ property; we prove that this property holds for locally  strongly regular graphs under various purely  combinatorial assumptions. We then classify the locally finite,  locally connected  $4$-CH graphs. We also classify the locally finite, locally disconnected $4$-CH graphs containing $3$-cycles and induced $4$-cycles, and 
 prove that, with the possible exception of   locally disconnected graphs
 containing $3$-cycles but no induced $4$-cycles,   
every finite $7$-CH graph is CH.
\end{abstract}

\section{introduction}
\label{s:intro}

A (simple undirected) graph $\Gamma$ is \textit{homogeneous} if any isomorphism between finite induced subgraphs extends to an automorphism of $\Gamma$. 
 An analogous definition can be made for any relational structure, and the study of these highly symmetric objects dates back to Fra\"iss\'e~\cite{Fra1953}. 
The finite and countably infinite homogeneous graphs have been classified~\cite{Gar1976,LacWoo1980,GolKli1978}, and very few families of graphs arise  (see Theorem \ref{thm:hom}). Consequently, various relaxations of homogeneity have been considered. For example,  a graph $\Gamma$ is $k$\textit{-homogeneous} if $k$ is a positive integer and  any isomorphism between induced subgraphs of order at most $k$ extends to an automorphism of $\Gamma$. Every locally finite $5$-homogeneous graph is homogeneous~\cite{Cam1980}---remarkably, this result does not rely upon the classification of the finite simple groups (CFSG)---but for each $k$,  there are uncountably many  countable $k$-homogeneous graphs that are not $(k+1)$-homogeneous~\cite{DroMac1991}. Further, for $2\leq k\leq 4$, the locally finite $k$-homogeneous graphs  have been classified using the CFSG~\cite{Buc1980,CamMac1985,LieSax1986} (see \S\ref{s:results}).
 
We require the following definition: for a graph (or graph property) $X$, we say that  a graph $\Gamma$ is \textit{locally X}  if the neighbourhood of any vertex in $\Gamma$   is non-empty and induces a graph that is  isomorphic to (or has  property) $X$;  see \S\ref{s:defn} for other unexplained terms. 

Here is another way to relax the concept of homogeneity: a graph is $k$\textit{-connected-homogeneous}, or $k$\textit{-CH}, if $k$ is a positive integer and  any isomorphism between connected induced subgraphs  of  order at most $k$ extends to an automorphism of the graph, and \textit{connected-homogeneous}, or \textit{CH}, if it is $k$-CH for all $k$. 
 The locally finite CH graphs have been classified~\cite{Gar1978,Eno1981} (see Theorem~\ref{thm:CH}), as have the countably infinite CH graphs~\cite{GraMac2010}, and Gray~\cite{Gra2009}  proved  that any infinite locally finite $3$-CH graph with more than one end is CH. 
The $1$-CH  and $2$-CH graphs are precisely the vertex-transitive and regular arc-transitive graphs, respectively.
The $3$-CH graphs  with girth at least~$4$ are  $2$-arc-transitive, and it is infeasible to classify such graphs in general, even in the finite case.  Indeed, for a classification of the  finite $2$-arc-transitive graphs,  dealing with just one of the cases arising from the reduction of  the third author in \cite[Theorem~2 and Corollary~4.2]{Pra1993} would require a classification of all transitive actions of all finite simple groups for which a point stabiliser acts $2$-transitively on one of its orbits. 
Finite $3$-CH graphs with girth $3$ are studied using different  terminology in~\cite{BamDevFawPra2015};  see also \cite{LiZhoSub}.   In this paper, we investigate locally finite $k$-CH graphs for~$k\geq 4$.

\begin{remark}
\label{remark:connected}
A  graph is $k$-CH if and only if it  is a disjoint union of connected $k$-CH graphs, all of which are isomorphic. Thus we will often restrict our attention to connected graphs.
\end{remark}

The study of $k$-CH graphs  (with non-zero valency) naturally divides into  the locally connected case and the locally disconnected case.   First we consider the locally connected case. (Definitions of the graphs that arise below may be found in \S\ref{s:graphs}.)

\begin{thm}
\label{thm:connected}
Let $\Gamma$ be a locally finite, connected, locally connected graph.  If $\Gamma$ is $4$-CH, then one of the following holds.
\begin{itemize}
\item[(i)] $\Gamma$ is  $K_n$ where $n\geq 2$ or  $K_{m[r]}$ where $m\geq 3$ and $r\geq 2$. Here $\Gamma$ is homogeneous.
\item[(ii)] $\Gamma$ is the Schl\"{a}fli graph. Here $\Gamma$ is $4$-homogeneous but not $5$-CH.
\item[(ii)] $\Gamma$ is the  McLaughlin graph. Here  $\Gamma$ is $4$-CH but not $5$-CH.
\end{itemize}
\end{thm}

\begin{cor}
\label{cor:connected}
Any  locally finite, connected, locally connected $5$-CH graph  is homogeneous.
\end{cor}

Any   disconnected $2$-homogeneous graph must be a disjoint union of  complete graphs with the same order (see Lemma \ref{lemma:2homsimple}). 
 Since a $3$-CH graph is locally $2$-homogeneous (see Lemma \ref{lemma:localhom}), it follows that a locally finite $3$-CH graph  is locally disconnected if and only if it is locally $(t+1)\cdot K_s$ for some positive integers $t$ and $s$. 
For such graphs, one important parameter is the number of common neighbours of two vertices at distance two; this number is constant for a $3$-CH graph~$\Gamma$, and we denote it by $c_2(\Gamma)$, or $c_2$ when context permits. Note that a  locally $(t+1)\cdot K_s$ graph  has girth~$3$ precisely when $s>1$,  and girth at least~$5$ precisely when $s=c_2=1$. 
Further, it has no induced $4$-cycles precisely when $c_2=1$ (see Lemma~\ref{lemma:notinduced}).
 We divide the locally disconnected case  into four cases: girth~$3$ with $c_2>1$, girth~$4$, girth at least~$5$, and girth~$3$ with $c_2=1$. In our next result, we classify the locally finite $4$-CH graphs in the first of these cases.

\begin{thm}
\label{thm:girth3quad}
Let $\Gamma$ be a locally finite, connected, locally disconnected graph with girth $3$ for which   $c_2>1$. If $\Gamma$ is $4$-CH, then one of the following holds.
\begin{itemize}
\item[(i)] $\Gamma$ is  $K_{n}\Box K_{n}$ where $n\geq 3$. Here $\Gamma$ is CH.
\item[(ii)] $\Gamma$ is the point graph of the generalised quadrangle $Q_4(3)$, $Q^-_5(3)$ or $Q^-_5(4)$. Here $\Gamma$ is $4$-CH but not $5$-CH.
\item[(iii)] $\Gamma$ is the point graph of the generalised quadrangle $Q^-_5(2)$. Here $\Gamma$ is $5$-CH but not~$6$-CH.
\end{itemize}
\end{thm}

Note that the point graph of $Q^-_5(2)$  and its complement the Schl\"{a}fli graph, which arose in the locally connected case,  are the only two locally finite $4$-homogeneous graphs that are not homogeneous~\cite{Buc1980}.  The following is an immediate consequence of Corollary~\ref{cor:connected} and Theorem~\ref{thm:girth3quad}.

\begin{cor}
\label{cor:girth3quad}
Any locally finite $6$-CH graph  with girth $3$ and $c_2>1$ is CH.
\end{cor}

For the three remaining cases, we do not have classifications of the locally finite $4$-CH graphs. Instead, we consider finite $k$-CH graphs for slightly larger $k$.  For the girth $4$ case, we use the classification of the  finite $4$-transitive permutation groups (a well-known consequence of the CFSG)  together with some results from~\cite{Cam1975} to prove the following; note that  families of finite CH graphs with girth $4$ do exist (see Theorems~\ref{thm:CH} or~\ref{thm:girth4plus}). 

 \begin{thm}
 \label{thm:girth4}
 \begin{itemize}
 \item[(i)]  Any finite $5$-CH graph with girth  $4$ is CH.
 \item[(ii)] There are infinitely many finite connected $4$-CH graphs with girth $4$ that are not $5$-CH.
 \end{itemize}
\end{thm}

For $k\geq 4$, finite $k$-CH graphs  with girth at least $5$ are $s$-arc-transitive for some $s\geq 3$, and such graphs have been studied extensively. In particular, Weiss \cite{Wei1981} proved that if $\Gamma$ is a finite $s$-arc-transitive graph with valency at least $3$, then $s\leq 7$, and if $s=7$, then $\Gamma$ has valency $3^e+1$ for some  positive integer $e$. Further,  Conder and Walker~\cite{ConWal1998} have constructed  infinitely many finite connected quartic $7$-arc-transitive graphs. Using these results, as well as the classification of the finite $4$-transitive permutation groups, we obtain our next theorem; see Theorems~\ref{thm:girth5diam2} and~\ref{thm:girth5strong} for more details.   Note that finite CH graphs with girth at least $5$ do exist: these include  the Petersen graph and any  cycle with  at least $5$ vertices.
 
 \begin{thm}
 \label{thm:girth5}
 \begin{itemize}
 \item[(i)] Any finite $7$-CH graph with girth at least $5$ is $CH$.
 \item[(ii)]  There are infinitely many finite connected quartic $6$-CH  graphs  with girth at least $12$ that are not $7$-CH.
 \item[(iii)]  A finite quartic graph with girth at least $7$ is $6$-CH if and only if it is $7$-arc-transitive.
 \end{itemize}
\end{thm}

The case where $\Gamma$ is locally disconnected  with  girth $3$ but $c_2=1$ seems to be more difficult. Here $\Gamma$ is locally $(t+1)\cdot K_s$ where $t\geq 1$ and $s\geq 2$. When $t=1$, we can make some progress, for   $\Gamma$ is the line graph of  a  graph $\Sigma$ with valency $s+1$ and  girth at least $5$, and it turns out that $\Gamma$ is $k$-CH precisely when $\Sigma$ is $(k+1)$-CH with  girth at least $k+2$ (see Lemma \ref{lemma:linegraphCH}). Thus results about $k$-CH graphs with girth at least $5$ 
can be interpreted for locally $2\cdot K_s$ graphs with $c_2=1$. In particular, we obtain the following; see Theorem~\ref{thm:2Ksplus} for more details. Note that the line graph of the regular tree of valency $s+1$ is an infinite CH locally $2\cdot K_s$ graph with $c_2=1$.

\begin{thm}
\label{thm:2Ks}
\begin{itemize}
\item[(i)] For $s\geq 2$, there are no finite $6$-CH locally $2\cdot K_s$ graphs with $c_2=1$.
\item[(ii)] There are infinitely many finite connected $5$-CH locally $2\cdot K_3$ graphs with $c_2=1$.
\end{itemize}
\end{thm}

Thus, with the possible exception of locally $(t+1)\cdot K_s$ graphs where $t\geq 2$, $s\geq 2$ and $c_2=1$, there exists an absolute constant $A$ such that every  finite $A$-CH graph is CH, and $A=7$ is the best possible constant. We summarise this result here.

\begin{cor}
\label{cor:all}
For a finite  $7$-CH graph $\Gamma$, one of the following holds.
\begin{itemize}
\item[(i)]  $\Gamma$ is locally $m\cdot K_n$ for some $m\geq 3$ and $n\geq 2$, and $c_2=1$.
\item[(ii)] $\Gamma$ is CH.
\end{itemize}
 Further, there are infinitely many finite connected $6$-CH graphs that are not $7$-CH.
\end{cor}

For $k$-CH graphs ($k\geq 4$) that satisfy the condition of Corollary \ref{cor:all}(i), we only have partial results; see Proposition~\ref{prop:badsmallval} and Lemmas~\ref{lemma:c2=1} and~\ref{lemma:cyclebig}. 
In particular, there are examples of $4$-CH graphs in this case: the point graph of the dual of the split Cayley hexagon of order $(2,2)$ is $4$-CH but not $5$-CH, as is 
the  point graph of the Hall-Janko near octagon. This leads us to the following open problem; note that there are infinite locally finite CH graphs that satisfy the condition of Corollary~\ref{cor:all}(i), but there are no such finite graphs  (see Theorem~\ref{thm:CH}).

\begin{problem}
Determine whether there exists an absolute constant $A$ for which there are no finite $A$-CH graphs satisfying the condition of Corollary $\ref{cor:all}$(i).
\end{problem}

Our approach for the locally connected case is to  use the classification of the finite $3$-homogeneous graphs (see Theorem \ref{thm:3hom}) together with the observation that any  $4$-CH graph is locally a $3$-homogeneous graph (see Lemma~\ref{lemma:localhom}). Further, we will show that there is a combinatorial property of graphs---the unique $x$ property---that acts as an obstruction for $4$-connected-homogeneity (see  Definition~\ref{defn:x} and Lemma~\ref{lemma:xplus});  we will prove that this property holds for 
 locally  strongly regular graphs under various   combinatorial assumptions   (see~\S\ref{s:uniquex}).

This paper is organised as follows. In \S\ref{s:prelim}, we provide some notation and definitions (\S\ref{s:defn}-\ref{s:graphs}), state some basic  results (\S\ref{s:other}), and state some classification theorems (\S\ref{s:results}). In \S\ref{s:kCH}, we establish some properties of $k$-CH graphs, and in \S\ref{s:uniquex}, we consider graphs with the unique $x$ property. In \S\ref{s:LC}, we consider locally connected graphs and prove Theorem \ref{thm:connected}; in \S\ref{s:GQ}, we consider locally disconnected graphs with girth $3$ and $c_2>1$ and prove Theorem \ref{thm:girth3quad}; in \S\ref{s:girth4}, we consider graphs with girth $4$ and prove Theorem \ref{thm:girth4}; in \S\ref{s:girth5}, we consider graphs with girth at least $5$ and prove Theorem \ref{thm:girth5};
 and in \S\ref{s:LDbad}, we consider locally disconnected graphs with girth $3$ and $c_2=1$ and prove Theorem \ref{thm:2Ks}.  Note that the proofs of  our main results depend upon the CFSG.
 
\begin{remark}
\label{remark:magma}
 We sometimes use {\sc  Magma}~\cite{Magma}   to determine whether a graph is  $k$-CH. These computations are routine: we construct   a $G$-arc-transitive graph $\Gamma$ using a   representation of a group $G$  provided by~\cite{web-atlas} or standard techniques, and we analyse $\Gamma$ using Lemma~\ref{lemma:detkCH}. This analysis could be performed using various  software packages; we find {\sc  Magma} the most convenient, and there is extensive online documentation  for the commands we use to implement Lemma~\ref{lemma:detkCH}.
\end{remark}

\section{Preliminaries}
\label{s:prelim}

All graphs in this paper are  undirected,  simple (no multiple
edges or loops) and have non-empty vertex sets, but they need not be finite or even locally finite. Basic graph theoretical terminology not given here may be found in the appendix of~\cite{BroCohNeu1989}. All group actions and graph isomorphisms are written on the right, and basic group theoretic terminology may be found in~\cite{Cam1999}. The notation used to denote the
finite simple groups (and their automorphism groups) is consistent with~\cite{KleLie1990}.

\subsection{Notation and definitions}
\label{s:defn}

 A  graph $\Gamma$ consists of a non-empty vertex set $V\Gamma$ and an edge set $E\Gamma$, which is a set of $2$-subsets of $V\Gamma$.  The \textit{order} of $\Gamma$ is $|V\Gamma|$. For a non-empty subset $X$ of $V\Gamma$, we often abuse notation and write $X$ for the subgraph of $\Gamma$ induced by $X$.   The \textit{girth} of a graph $\Gamma$ is the length of a shortest cycle in $\Gamma$ (or infinity when $\Gamma$ has no cycles). We write $\overline{\Gamma}$ for the complement of $\Gamma$. When $E\Gamma$ is non-empty, the \textit{line graph} of $\Gamma$, denoted by $L(\Gamma)$, has vertex set $E\Gamma$, and two vertices of $L(\Gamma)$ are adjacent whenever the corresponding edges of $\Gamma$ have a common vertex in $\Gamma$. We denote the distance between $u,v\in V\Gamma$ by
$d_\Gamma(u,v)$, and the diameter of $\Gamma$  by $\diam(\Gamma)$. When $\Gamma$ is connected and  bipartite, a \textit{halved graph} of $\Gamma$ is one of the two connected components of $\Gamma_2$, where $\Gamma_2$ is the graph with  vertex set $V\Gamma$  in which  $u$ and $v$  
are adjacent whenever $d_\Gamma(u,v)=2$.
For $u\in V\Gamma$ and any integer $i\geq 0$, let $\Gamma_i(u):=\{v\in
V\Gamma:d_\Gamma(u,v)=i\}$. We write $\Gamma(u)$ for the \textit{neighbourhood} $\Gamma_1(u)$. The cardinality of $\Gamma(u)$ is the \textit{valency} of $u$, and the graph induced by $\Gamma(u)$ (when non-empty) is a \textit{local graph}  of $\Gamma$. When every vertex of $\Gamma$ has the same valency, we say that  $\Gamma$ is \textit{regular} and refer to the \textit{valency} of $\Gamma$. A graph is \textit{locally finite} if every vertex has finite valency; this extends the definition given in the introduction to include graphs with valency $0$ (such as $K_1$). Note that, by definition,  a graph with valency $0$  is neither locally connected nor locally disconnected.
 
For a positive integer $s$, a \textit{path of length} $s$ in $\Gamma$ is a sequence of vertices $(u_0,\ldots,u_s)$ such that $u_i$ is adjacent to $u_{i+1}$ for $0\leq i<s$,   an \textit{s-arc} is a path $(u_0,\ldots,u_s)$ where $u_{i-1}\neq u_{i+1}$ for $0< i< s$, and an \textit{arc} is a $1$-arc. An  ($s$\textit{-})\textit{geodesic}  is a path $(u_0,\ldots,u_s)$ where $d_\Gamma(u_0,u_s)=s$. A \textit{path graph} is a tree $T$ with $|T(u)|\leq 2$ for all $u\in VT$. A $\mu$\textit{-graph} of $\Gamma$ is a graph that is induced by $\Gamma(u)\cap \Gamma(w)$ for some $u,w\in V\Gamma$ with $d_\Gamma(u,w)=2$. 

For this paragraph, assume that $\Gamma$ is locally finite, and let $i$ be a non-negative integer.  We write $k_i(\Gamma)$ or $k_i$ for $|\Gamma_i(u)|$ whenever $|\Gamma_i(u)|$ does not depend on the choice of $u$,    or to indicate that we are assuming this. 
For $u,v\in V\Gamma$ such that $d_\Gamma(u,v)=i$, let $c_i(u,v) :=|\Gamma_{i-1}(u)\cap \Gamma(v)|$ (when $i\geq 1$),  $a_i(u,v) :=|\Gamma_{i}(u)\cap \Gamma(v)|$ and $b_i(u,v)  :=|\Gamma_{i+1}(u)\cap \Gamma(v)|$. Whenever $c_i(u,v)$  does not depend on the choice of $u$ and $v$,  or to indicate that we are assuming this, 
 we write $c_i(\Gamma)$ or $c_i$, and similarly for $a_i(u,v)$ and $b_i(u,v)$. A graph $\Gamma$ is \textit{strongly regular} with parameters $(v,k,\lambda,\mu)$ if it is finite with order $v$ and regular with valency $k$ where $0<k<v-1$, and if any two adjacent vertices have $\lambda=\lambda(\Gamma)$ common neighbours and any two non-adjacent vertices have $\mu=\mu(\Gamma)$ common neighbours. Note that  complete graphs and edgeless graphs are not strongly regular. The complement of a strongly regular graph is again strongly regular with parameters $(v,v-k-1,v-2k+\mu-2,v-2k+\lambda)$.

Let $G$ be a group acting on a set $\Omega$. We denote the permutation group induced by this action  by $G^\Omega$, and  the pointwise stabiliser in $G$ of $u_1,\ldots,u_n\in \Omega$ by $G_{u_1,\ldots,u_n}$. When $G$ is transitive on $\Omega$, any orbit of $G$ on $\Omega\times\Omega$ is an  \textit{orbital}; the \textit{trivial} orbital is $\{(u,u): u\in \Omega\}$, and the other orbitals are \textit{non-trivial}.   The \textit{rank} of $G$ is the number of orbitals.  Equivalently, for any $u\in \Omega$, the rank of $G$ is the number of orbits of $G_u$ on $\Omega$. The group $G$ is \textit{primitive} if it is transitive and there are no non-trivial $G$-invariant equivalence relations  on $\Omega$ (the trivial ones are $\{(u,u):u\in \Omega\}$ and $\Omega\times\Omega$).
The group $G$ is $k$\textit{-transitive} if $1\leq k\leq |\Omega|$ and $G$ acts transitively on the set of $k$-tuples of pairwise distinct elements of $\Omega$.

A graph $\Gamma$ is
$G$\textit{-vertex-transitive} (or $G$\textit{-arc-transitive})  if $G$ is a subgroup of the automorphism group $ \Aut(\Gamma)$ and $G$  acts transitively on $V\Gamma$ (or the arcs of $\Gamma$); we omit the prefix $G$ when $G=\Aut(\Gamma)$.
  We say that $\Gamma$ is   $s$\textit{-arc-transitive} if  $\Aut(\Gamma)$ acts transitively on the set of $s$-arcs, and  $s$\textit{-transitive} if $\Gamma$ is $s$-arc-transitive but not $(s+1)$-arc-transitive. (Note that we have defined two different concepts of $n$-transitivity: one for groups above, and one for graphs here.) When  $\Aut(\Gamma)$ acts transitively on ordered pairs of vertices at distance $i$ for each integer $i\geq 0$, we say that $\Gamma$ is \textit{distance-transitive}, and when $\Gamma$ is also finite with diameter $d$, it has \textit{intersection array} $\{b_0,b_1,\ldots, b_{d-1};c_1,c_2,\ldots,c_d\}$. 
We say that $\Gamma$ is $(G,k)$\textit{-homogeneous} (or $(G,k)$\textit{-CH}) when $k$ is a positive integer,   $G\leq \Aut(\Gamma)$,  and any isomorphism between (connected) induced subgraphs of $\Gamma$  with order at most $k$ extends to an element of $G$. Note that a $(G,2)$-CH graph is $G$-vertex-transitive and $G$-arc-transitive.

 A \textit{partial linear space} is a pair $(\mathcal{P},\mathcal{L})$  where $\mathcal{P}$ is a non-empty set of \textit{points} and  $\mathcal{L}$  is a collection of subsets of $\mathcal{P}$ called \textit{lines} such that two distinct points are in at most one line, and every line contains at least two points.   The \textit{incidence graph} of a partial linear space  is the bipartite graph whose vertices are the points and lines, where a point $p$ is adjacent to a line $\ell$ whenever $p\in \ell$. The \textit{point graph} of a partial linear space is the graph whose vertices are the points, where two points are adjacent whenever they are collinear. For a positive integer $n$, a \textit{generalised n-gon} is a partial linear space whose incidence graph $\Gamma$ has diameter $n$  with  $c_i=1$ 
 (i.e., $c_i(u,v)=1$ for all $u,v\in V\Gamma$ such that $d_\Gamma(u,v)=i$) for $i<n$.
 A generalised $n$-gon is \textit{thick} when every line has size at least three and every point is contained in at least three lines, and it has \textit{order} $(s,t)$ when there exist positive integers $s$ and $t$  such that every line has size $s+1$ and every point is contained in $t+1$ lines.  It is routine to prove that  any thick generalised $n$-gon has order $(s,t)$ for some $s$ and $t$.  A generalised $n$-gon is \textit{distance-transitive} if its point graph is distance-transitive.  When $n=4$, $6$ or $8$, a generalised $n$-gon is a \textit{generalised quadrangle}, \textit{generalised hexagon} or \textit{generalised octagon} respectively. A generalised quadrangle satisfies the \textit{GQ Axiom}: for each point $p$ and line $\ell$ such that $p\notin\ell$, there is a unique $q\in \ell$ such that $p$ is collinear with $q$. Conversely, any partial linear space with at least two lines that satisfies the GQ Axiom is a generalised quadrangle. See \cite[\S 6.5]{BroCohNeu1989} for more details. 
 
 We denote a finite field of order $q$ by $\mathbb{F}_q$ and a $d$-dimensional vector space over $\mathbb{F}_q$ by $V_d(q)$. We will use the following terminology concerning forms. See \cite[\S 2.1,2.3-2.5]{KleLie1990} for  more information.  A \textit{symplectic}, \textit{unitary} or \textit{quadratic} space is a pair $(V_d(q),\kappa)$ where $\kappa$ is, respectively, a non-degenerate symplectic, unitary or quadratic form on $V_d(q)$. In a symplectic or unitary space $(V,f)$,   a vector $v\in V$ is \textit{singular} if $f(v,v)=0$, and in a quadratic space  $(V,Q)$, a vector $v\in V$ is \textit{singular} if  $Q(v)=0$. In a symplectic, unitary or quadratic space $(V,\kappa)$, a subspace $W$ of $V$ is \textit{totally singular} if every vector in $W$ is singular. For a positive integer $m$, a quadratic space $(V_{2m}(q),Q)$ has \textit{plus type}  when the maximal totally singular subspaces of $V_{2m}(q)$ have dimension $m$,  and  \textit{minus type}  when the maximal totally singular subspaces of $V_{2m}(q)$ have dimension $m-1$; we also say that the quadratic space has type $\varepsilon$ where $\varepsilon\in \{+,-\}$. 
 
\subsection{Families of graphs}
\label{s:graphs}

Let $m$ and $n$ be   positive integers. We denote the complete graph with $n$ vertices  by $K_n$, the cycle with $n$ vertices by $C_n$, the complete multipartite graph with $n$ parts of size $m$  by
$K_{n[m]}$, and its complement (the disjoint union of $n$ copies of $K_m$) by $n\cdot K_m$. 
We also write $K_{n,n}$ for the complete bipartite graph $K_{2[n]}$. The \textit{grid graph} $K_n\Box K_m$ has vertex set $VK_n\times VK_m$, where distinct vertices $(u_1,u_2)$ and $(v_1,v_2)$ are adjacent whenever $u_1=v_1$ or $u_2=v_2$. We denote the complement of this graph by $K_n\times K_m$. 
In the literature, the graph $K_2\times K_{n}$ is often described as the graph $K_{n,n}$ with the edges of a perfect matching removed. The \textit{n-cube} $Q_n$ has vertex set $\mathbb{F}_2^n$, where two vertices are adjacent whenever they differ in exactly one coordinate. The \textit{folded n-cube} $\Box_n$ is obtained from $Q_n$ by identifying those vertices $u$ and $v$  for which $u+v=(1,\ldots,1)$. The \textit{affine polar graph} $VO^\varepsilon_{2m}(q)$ has vertex set $V_{2m}(q)$,  and  vectors $u$ and $v$ are adjacent whenever $Q(u-v)=0$, where $(V_{2m}(q),Q)$ is  a quadratic space with type  $\varepsilon$.  

We will be interested in the point or incidence graphs of the following classical generalised quadrangles where $q$ is a power of a prime: $W_3(q)$, $H_3(q^2)$, $H_4(q^2)$, $Q_4(q)$ for $q$ odd, and $Q_5^-(q)$. These are defined as follows. In each case, we have a symplectic, unitary or quadratic space $(V_d(s),\kappa)$, 
the points are the one-dimensional totally singular subspaces of $V_d(s)$ with respect to $\kappa$, and the lines are the two-dimensional totally singular subspaces (which we may of course view as sets of points). For $W_3(q)$, we take a symplectic space with $(d,s)=(4,q)$; for $H_3(q^2)$ or $H_4(q^2)$, a unitary space with $(d,s)=(4,q^2)$ or $(5,q^2)$ respectively; for $Q_4(q)$, a quadratic space with $(d,s)=(5,q)$ and $q$ odd; and for $Q_5^-(q)$, a quadratic space of minus type  with $(d,s)=(6,q)$. 

 We will also be interested in certain other   generalised $n$-gons. The \textit{split Cayley hexagon} is a generalised hexagon of order $(q,q)$ whose automorphism group contains the exceptional group of Lie type $G_2(q)$.  
The \textit{Ree-Tits octagon} is  a generalised octagon  of order $(q,q^2)$ whose automorphism group contains the exceptional  group of Lie type $^2F_4(q)$. See~\cite{Van2012} 
 for more details.

The \textit{Clebsch graph} is the halved $5$-cube (see~\S\ref{s:defn} for the definition of a halved graph). We caution the reader that some authors define the Clebsch graph to be the complement of the halved $5$-cube, which  is isomorphic to  $\Box_5$ and  $VO_4^-(2)$;  our definition is consistent with~\cite{BroCohNeu1989} and Seidel~\cite{Sei1968}. The \textit{Petersen graph} is the complement of the local graph of the halved $5$-cube. The \textit{Higman-Sims} graph is a strongly regular graph with parameters $(100,22,0,6)$ whose automorphism group is $\HS{:}2$, where  $\HS$ denotes the Higman-Sims group, a sporadic simple group (see \cite[\S 13.1B]{BroCohNeu1989}). Similarly, the \textit{McLaughlin graph} is a strongly regular graph with parameters $(275,112,30,56)$ whose automorphism group is $\McL{:}2$, where  $\McL$ denotes the McLaughlin group, another sporadic simple group (see \cite[\S 11.4H]{BroCohNeu1989}). The \textit{Schl\"{a}fli graph} is the complement of the point graph of the generalised quadrangle $Q_5^-(2)$. 

For positive integers $t$ and $s$, the \textit{biregular tree} $T_{t+1,s+1}$ is an (infinite) tree with  bipartition $(V_t,V_s)$ such that the vertices in $V_t$ have valency $t+1$, and the vertices in $V_s$ have valency $s+1$. The halved graph of $T_{t+1,s+1}$ with vertex set $V_t$ is  locally  $(t+1)\cdot K_s$, while the halved graph with vertex set $V_s$ is locally $(s+1)\cdot K_t$.  We will see in \S\ref{s:results} that the halved graphs of $T_{t+1,s+1}$ for $t,s\geq 1$ are precisely  the infinite locally finite connected CH  graphs. To obtain the complete list of such graphs, it suffices to either consider only those halved graphs with vertex set $V_t$, or  assume that $t\geq s$, but we allow this redundancy in the notation for simplicity. 

\subsection{Basic results} 
\label{s:other}

Almost by definition,  we obtain the following  useful observation concerning $2$-homogeneous graphs; see \cite[Lemma~1]{Gar1976}.

\begin{lemma}[\cite{Gar1976}]
\label{lemma:2homsimple}
A $2$-homogeneous graph is either a disjoint union of complete graphs with the same  order, or connected with diameter $2$.
\end{lemma}

Note that we permit the disjoint union to contain only one complete graph.  The following is immediate from Lemma~\ref{lemma:2homsimple}.

\begin{lemma}
\label{lemma:2homsimplelf}
If $\Gamma$ is a locally finite $2$-homogeneous graph, then either $\Gamma$ is a (possibly infinite) disjoint union of finite complete graphs with the same  order, or $\Gamma$ is finite with diameter $2$.
\end{lemma}

Note that Lemma \ref{lemma:2homsimplelf} is often used to derive a result about locally finite 2-homogeneous graphs from the analogous result for the finite case. 

We will sometimes use a stronger form of Lemma~\ref{lemma:2homsimple}. Recall the definition of a $(G,k)$-homogeneous graph from \S\ref{s:defn}. If $\Gamma$ is a non-complete $(G,2)$-homogeneous graph that contains an edge, then the non-trivial orbitals of $G$ on $V\Gamma$ 
are the sets of adjacent pairs and distinct non-adjacent pairs, so $G$ is transitive of rank $3$ on $V\Gamma$. If $G$ also preserves a non-trivial  equivalence relation $\equiv$ on $V\Gamma$, then the set of pairs of distinct vertices $u$ and $v$ such that $u\equiv v$ must be either the set of adjacent pairs, or the set of distinct non-adjacent pairs, so $\Gamma$ is either a disjoint union of complete graphs, or a complete multipartite graph. Thus we have the following result.

\begin{lemma}
\label{lemma:2hom}
Let $\Gamma$ be a $(G,2)$-homogeneous graph. Then  exactly one of the following holds.
\begin{itemize}
\item[(i)] $\Gamma$ or $\overline{\Gamma}$ is a disjoint union of complete graphs with the same  order.
\item[(ii)] $\diam(\Gamma)=2$ and $G$ is primitive of rank $3$ on $V\Gamma$.
\end{itemize}
\end{lemma}

Next we give a sufficient condition for the local action of an arc-transitive graph to be faithful; this will be useful in the locally connected case.

\begin{lemma}
\label{lemma:faithful}
 Let $\Gamma$ be a connected  $G$-arc-transitive graph. If there exists $x\in V\Gamma$ and $y\in\Gamma(x)$ such that the pointwise stabiliser of $\Gamma(x)\cap \Gamma(y)$ in $G_{x,y}$ also fixes $\Gamma(x)$ pointwise, then the action of $G_u$ on $\Gamma(u)$ is faithful for all $u\in V\Gamma$.
\end{lemma}

\begin{proof}
Suppose that there exists $x\in V\Gamma$ and $y\in\Gamma(x)$ such that the pointwise stabiliser of $\Gamma(x)\cap \Gamma(y)$ in $G_{x,y}$ also fixes $\Gamma(x)$ pointwise. Since $G$ acts transitively on the arcs of $\Gamma$, it follows that for any $u\in V\Gamma$ and $v\in \Gamma(u)$,  if $h\in G_{u,v}$ fixes $\Gamma(u)\cap \Gamma(v)$ pointwise, then $h$ fixes $\Gamma(u)$ pointwise. 
  Suppose that $g\in G_u$ fixes $\Gamma(u)$ pointwise. Let $w\in V\Gamma$. There is a path $(u_0,\ldots,u_\ell)$ in $\Gamma$ with $u_0=u$ and $u_\ell=w$. If $g$ fixes $\{u_i\}\cup \Gamma(u_i)$ pointwise for some integer $i\geq 0$, then $g$ fixes $u_{i+1}$, $u_i$ and $\Gamma(u_{i+1})\cap \Gamma(u_{i})$ pointwise, so, by the above observation, $g$ fixes  $\{u_{i+1}\}\cup\Gamma(u_{i+1})$ pointwise. By induction, $g$ fixes $u_\ell=w$. Thus $g=1$.
\end{proof}

The proof of the following is routine.

\begin{lemma}
\label{lemma:notinduced}
Let $\Gamma$ be a locally $(t+1)\cdot K_s$ graph for positive integers $t$ and $s$. Then no induced subgraph of $\Gamma$ is isomorphic to the complete graph $K_4$ with one edge removed. 
\end{lemma}

The next result is a well-known property of quadratic spaces; see the proof of~\cite[Proposition~2.5.3]{KleLie1990}, for example.

\begin{lemma}
\label{lemma:quadratic}
Let $(V,Q)$ be a quadratic space, and let $f$ be  the bilinear form associated with $Q$. For any  non-zero singular vector $v$, there exists $w\in V$ such that $Q(w)=0$ and $f(v,w)=1$.
\end{lemma}

We finish this section with a result of Tutte; see~\cite[Lemma 4.1.3]{GodRoy2001}  for a proof.

\begin{lemma}
\label{lemma:girth}
An $s$-arc-transitive graph with valency at least $3$ has girth at least $2s-2$.
\end{lemma}

\subsection{Classification theorems}
\label{s:results}

To begin, we state the classification of the finite homogeneous graphs; this was obtained independently by Gardiner~\cite{Gar1976} (using  the  work of Sheehan~\cite{She1974}) and Gol'Fand and Klin~\cite{GolKli1978}. In fact, this classification  immediately implies  that the locally finite homogeneous graphs are known (as Gardiner notes in \cite{Gar1978}) since such  graphs are either disjoint unions of complete graphs with the same order, or finite  with diameter $2$ (see Lemma~\ref{lemma:2homsimplelf}). However, we will only state the classification in the finite case  for simplicity. 

\begin{thm}[\cite{Gar1976,GolKli1978}]
\label{thm:hom}
A finite graph $\Gamma$ is homogeneous if and only if $\Gamma$ is listed below.
\begin{itemize}
\item[(i)] $(t+1)\cdot K_s$ where $t\geq 0$ and $s\geq 1$.
\item[(ii)] $K_{m[r]}$ where $m\geq 2$ and $r\geq 2$.
\item[(iii)] $C_5$ or $K_3\Box K_3$. 
\end{itemize}
\end{thm}

In the introduction, we stated two important results concerning $k$-homogeneous graphs: first, every  locally finite $5$-homogeneous graph is homogeneous~\cite{Cam1980}, and second, the only locally finite $4$-homogeneous graphs that are not $5$-homogeneous are the point graph of $Q_5^-(2)$ and its complement the Schl\"{a}fli graph~\cite{Buc1980}. We also alluded to the fact that  the locally finite $2$- and $3$-homogeneous  graphs are known. We now give some more details about these classifications.

By  Lemmas~\ref{lemma:2homsimplelf} and~\ref{lemma:2hom}, in order to classify the locally finite $2$-homogeneous graphs, it suffices to consider those finite graphs $\Gamma$ for which $\diam(\Gamma)=2$ and  $\Aut(\Gamma)$ is primitive of rank $3$. Using the CFSG, the finite primitive permutation groups of rank $3$ were classified in a series of papers (see \cite{LieSax1986}). A non-trivial orbital $X$ of such a group $G$ is the adjacency relation of a graph $\Gamma$ with $G\leq \Aut(\Gamma)$ precisely when $X$ is self-paired (i.e.,  symmetric), and this occurs precisely when $|G|$ is even.  Thus the  locally finite $2$-homogeneous graphs are known as an immediate consequence of the classification of the  finite primitive rank $3$ groups. 
Note that any  locally finite connected non-complete $2$-homogeneous graph  is finite of diameter $2$ and is therefore distance-transitive and strongly regular. 

The  locally finite $k$-homogeneous graphs for $k\geq 3$   are therefore also known, but these graphs were in fact enumerated (using the CFSG for $k\leq 4$) before 
 the classification of the finite primitive rank $3$ groups
was available. We now state Cameron and Macpherson's classification \cite{CamMac1985} of the  finite $3$-homogeneous graphs;  the locally finite classification then follows from Lemma~\ref{lemma:2homsimplelf}.

\begin{thm}[\cite{CamMac1985}]
\label{thm:3hom}
A finite graph $\Gamma$ is $3$-homogeneous if and only if $\Gamma$ or $\overline{\Gamma}$ is listed below.
\begin{itemize}
\item[(i)] $(t+1)\cdot K_s$ where $t\geq 0$ and $s\geq 1$.
\item[(ii)] $K_n\Box K_n$ where $n\geq 3$.
\item[(iii)]  $VO^\varepsilon_{2m}(2)$ where $m\geq 3$ and $\varepsilon\in\{+,-\}$.
\item[(iv)] The point graph of  $Q^-_5(q)$ where $q$ is a power of a prime.
\item[(v)] $C_5$, the  Clebsch graph, the Higman-Sims graph, or the McLaughlin graph.
\end{itemize}
\end{thm}

Our statement of Theorem \ref{thm:3hom} may appear to differ from  \cite[Corollary 1.2]{CamMac1985}, but it does describe the same set of graphs, and it is routine to verify that these graphs are indeed $3$-homogeneous.  In (iii), we impose the restriction $m\geq 3$ since   $VO^+_2(2)\simeq K_{2,2}$, $VO^-_2(2)\simeq \overline{K_4}$ and  $VO^+_4(2)\simeq K_4\times K_4$, all of which arise in (i) or (ii), and $VO^-_4(2)$ is isomorphic to  the complement of the Clebsch graph. The graphs we list in  (iv) are isomorphic to those of  \cite[Corollary 1.2]{CamMac1985}(iv) since    $Q_5^-(q)$ is isomorphic to the point-line dual of  $H_3(q^2)$.

Next we state the classification of the locally finite CH graphs. This classification implies, in particular, that the only infinite,  locally finite,  connected CH  graphs  are the halved graphs of the biregular tree $T_{t+1,s+1}$ for positive integers $t$ and $s$. Note that the halved graphs of $T_{t+1,2}$ are the regular tree $T_{t+1}$ and its line graph $L(T_{t+1})$.  Gardiner mistakenly  claimed in~\cite{Gar1978} that these are the only  infinite,  locally finite, connected CH graphs, but this was later corrected by Enomoto~\cite[Remark 3]{Eno1981} to include the  halved graphs of $T_{t+1,s+1}$ for $s\geq 2$. 

\begin{thm}[\cite{Gar1978,Eno1981}]
\label{thm:CH}
A locally finite connected graph $\Gamma$ is CH if and only if $\Gamma$ is listed below.
\begin{itemize}
\item[(i)] $K_{n}$ where $n\geq 1$ or $K_{m[r]}$ where $m\geq 2$ and $r\geq 2$.
\item[(ii)] $C_n$ where $n\geq 5$.
\item[(iii)] $K_n\Box K_n$ where $n\geq 3$.
\item[(iv)] $K_2\times K_{n}$ where $n\geq 4$.
\item[(v)] A  halved graph of the biregular tree $T_{t+1,s+1}$  where $t\geq 1$ and $s\geq 1$. 
\item[(vi)] The Petersen graph, or the folded $5$-cube $\Box_5$.
\end{itemize}
\end{thm}

The  finite distance-transitive generalised quadrangles were classified by Buekenhout and Van Maldeghem~\cite{BueVan1994} using the CFSG. In \S\ref{s:GQ}, we will use the following consequence of their work.  

\begin{thm}[\cite{BueVan1994}]
\label{thm:GQ}
Let $\mathcal{Q}$ be a finite thick distance-transitive generalised quadrangle of order $(s,t)$ where $s$ divides $t$. Let $G:=\Aut(\mathcal{Q})$. Then one of the following holds.
\begin{itemize}
\item[(i)] $\mathcal{Q}$ is $W_3(q)$ for a prime power $q$. Here $(s,t)=(q,q)$ and $G=\PGaSp_4(q)$.
\item[(ii)] $\mathcal{Q}$ is $Q_4(q)$ for an odd prime power $q$. Here $(s,t)=(q,q)$ and $G=\PGammaO_5(q)$.
\item[(iii)]$\mathcal{Q}$ is $Q_5^-(q)$ for a prime power $q$. Here $(s,t)=(q,q^2)$ and $G=\PGammaO^-_6(q)$. 
\item[(iv)] $\mathcal{Q}$ is $H_4(q^2)$ for a prime power $q$. Here $(s,t)=(q^2,q^3)$ and $G=\PGaU_5(q)$.
\end{itemize}
\end{thm}

\begin{proof}
By \cite{BueVan1994}, $\mathcal{Q}$ is one of $W_3(q)$, $Q_4(q)$ for $q$ odd, $Q_5^-(q)$, $H_3(q^2)$, $H_4(q^2)$, the dual of $H_4(q^2)$,  or a generalised quadrangle with order $(3,5)$. Note that $H_3(q^2)$ has order $(q^2,q)$, and the dual of $H_4(q^2)$ has order $(q^3,q^2)$.  Since $\mathcal{Q}$ has order $(s,t)$ where $s$ divides $t$, one of (i)-(iv) holds. 
\end{proof}

The following is  a well-known consequence of the CFSG (see \cite[Theorem 4.11]{Cam1980}). 

\begin{thm}[CFSG]
\label{thm:4trans}
The only  finite $4$-transitive permutation groups of degree $n$ are $S_n$ for $n\geq 4$, $A_n$ for $n\geq 6$, and the Mathieu groups $\M_n$ for $n\in \{11,12,23,24\}$.
\end{thm}

\section{Properties of $k$-CH graphs}
\label{s:kCH}

 We begin with a result that provides the basic approach for studying $k$-CH graphs in the locally connected case. Note that many of the arguments in this section apply to infinite graphs, including those that are not locally finite. Recall the definition of a $(G,k)$-CH graph from \S\ref{s:defn}.

\begin{lemma}
\label{lemma:localhom}
 If $\Gamma$ is a $(G,k)$-CH graph with non-zero valency for some $k\geq 2$, then for each $u\in V\Gamma$, the graph induced by $\Gamma(u)$ is $(G_u^{\Gamma(u)},k-1)$-homogeneous.
\end{lemma}

\begin{proof}
Let $u\in V\Gamma$. Let $\Delta_1$ and $\Delta_2$ be induced subgraphs of  $\Gamma(u)$ of order at most $k-1$, and suppose that $\varphi:\Delta_1\to\Delta_2$ is a graph isomorphism. For each $i$, let $\Sigma_i$ denote the subgraph of $\Gamma$ induced by $V\Delta_i\cup \{u\}$. Define $\varphi^*:\Sigma_1\to\Sigma_2$ by $u\mapsto u$ and $v\mapsto v\varphi$ for all $v\in V\Delta_1$. Now $\varphi^*$ is an isomorphism between connected graphs of order at most $k$, so there exists $g\in G$ such that $v^g=v\varphi^*$ for all $v\in V\Sigma_1$. Since $g\in G_u$, it preserves $\Gamma(u)$ and therefore induces an automorphism of $\Gamma(u)$ that extends  $\varphi$, as desired.
\end{proof}

\begin{lemma}
\label{lemma:3CH}
If $\Gamma$ is a locally finite $3$-CH graph with non-zero valency, then $\Gamma$ is either locally $(t+1)\cdot K_s$ for some $t\geq 0$ and $s\geq 1$, or locally a graph with diameter $2$.
\end{lemma}

\begin{proof}
Apply Lemmas~\ref{lemma:2homsimple} and~\ref{lemma:localhom}.
\end{proof}

 Our next result provides a method for  determining  whether a $(k-1)$-CH graph is $k$-CH. This requires some additional terminology: a graph $\Gamma$ is $(G,\Delta)$\textit{-homogeneous} if $G\leq \Aut(\Gamma)$ and $\Delta$ is a finite graph such that any isomorphism between induced subgraphs of $\Gamma$ that are isomorphic to $\Delta$ extends to an automorphism of $\Gamma$. Note that  $\Gamma$ is $(G,k)$-CH if and only if $\Gamma$ is $(G,\Delta)$-homogeneous for all connected graphs $\Delta$ of order at most $k$.
 
\begin{lemma}
\label{lemma:detkCH}
Let $\Gamma$ be a graph, let $\Delta$ be an induced subgraph of $\Gamma$ of order $k\geq 2$, and let $\Sigma$ be an induced subgraph of $\Delta$ of order $k-1$.
 Let $u_0$ be the unique vertex in $V\Delta\setminus V\Sigma$,   
 and let $X:=\{u\in V\Gamma\setminus V\Sigma:\Gamma(u)\cap V\Sigma=\Gamma(u_0)\cap V\Sigma\}$. If $\Gamma$ is $(G,\Sigma)$-homogeneous, then the following are equivalent.
\begin{itemize}
\item[(i)] $\Gamma$ is $(G,\Delta)$-homogeneous.
\item[(ii)] 
The pointwise stabiliser in $G$ of $V\Sigma$ 
is transitive on $X$.
\end{itemize}
\end{lemma}

\begin{proof}
Let $P$ be the pointwise stabiliser in $G$ of $V\Sigma$. First suppose that (ii) does not hold. Let $u_1$ and $u_2$ be elements of $X$ in different orbits of $P$. Let $\Delta_i$ be the graph induced by $V\Sigma\cup \{u_i\}$ for each $i$. Then $\Delta_1\simeq \Delta_2\simeq  \Delta$, and there is an isomorphism $\varphi:\Delta_1\to\Delta_2$  that fixes $V\Sigma$ pointwise and maps $u_1$ to $u_2$. Since  $\varphi$ cannot be extended to $G$, (i) does not hold.

Conversely, suppose that (ii) holds. 
Let $\varphi:\Delta_1\to\Delta_2$ be an isomorphism between induced subgraphs $\Delta_1$ and $\Delta_2$ of $\Gamma$ such that $\Delta_1$ is isomorphic to $\Delta$. There exists an isomorphism $\varphi_1:\Delta_1\to \Delta$. Hence $\varphi_2:=\varphi^{-1}\varphi_1:\Delta_2\to \Delta$ is also an isomorphism. Fix $i\in \{1,2\}$. Let $\Sigma_i:=\Sigma\varphi_i^{-1}$ and $w_i:=u_0\varphi_i^{-1}$, so that $V\Delta_i=V\Sigma_i\cup \{w_i\}$, and let $\varphi_i':=\varphi_i|_{V\Sigma_i}$. Now $\varphi_i':\Sigma_i\to \Sigma$ is an isomorphism of induced subgraphs of $\Gamma$ of order $k-1$, so there exists $g_i\in G$ that extends $\varphi_i'$. Let $u_i:=w_i^{g_i}$. 
 Now $u_i\in X$ since $\Gamma(u_0)\cap V\Sigma=(\Gamma(w_i)\cap V\Sigma_i)\varphi_i'=(\Gamma(w_i)\cap V\Sigma_i)^{g_i}=\Gamma(u_i)\cap V\Sigma$. By assumption, there exists $g\in P$ such that $u_1^g=u_2$. Further, $g$ extends the isomorphism $g_1^{-1}\varphi g_2:\Delta_1^{g_1}\to \Delta_2^{g_2}$ since $(V\Delta_i)^{g_i}=V\Sigma\cup\{u_i\}$. Thus $g_1gg_2^{-1}$ extends $\varphi$, and (i) holds.
\end{proof}

Note that for any connected graph $\Delta$ of order $k\geq 2$, we can always find a connected induced subgraph $\Sigma$ of $\Delta$ of order $k-1$:   choose $u,v\in V\Delta$ such that $d_\Delta(u,v)=\diam(\Delta)$ and take $\Sigma$ to be the graph induced by $V\Delta\setminus \{u\}$. 

The next two results  will be instrumental in our proof of Theorem \ref{thm:connected}.

\begin{lemma}
\label{lemma:x}
Let $\Gamma$ be a $(G,4)$-CH graph. If there exists $u\in V\Gamma$, $v\in\Gamma(u)$ and $w\in \Gamma_2(u)\cap \Gamma(v)$ such that $G_{u,v,w}$ fixes   some $x\in \Gamma(u)\cap\Gamma_2(v)$, then the following hold.
\begin{itemize}
\item[(i)] The graph induced by $\Gamma(u)\cap\Gamma_2(v)$ is either edgeless or  complete.
\item[(ii)] 
$\Gamma(u)\cap\Gamma_2(v)\cap\Gamma(w)$ is either $\{x\}$ or $(\Gamma(u)\cap\Gamma_2(v))\setminus \{x\}$.
\end{itemize}
\end{lemma}

\begin{proof}
Let $X:=\Gamma(u)\cap\Gamma_2(v)$. If $y,z\in X\cap \Gamma(w)$, then there is an isomorphism  between the connected graphs induced by $\{u,v,w,y\}$ and $\{u,v,w,z\}$ that maps $y$ to $z$ and fixes $u$, $v$ and $w$, so $G_{u,v,w}$ acts transitively on $X\cap \Gamma(w)$. Similarly,  $G_{u,v,w}$ acts transitively on  $X\setminus \Gamma(w)$. Thus $G_{u,v,w}$ has at most two orbits on $X$. By assumption, $G_{u,v,w}\leq G_{u,v,x}$, so $G_{u,v,x}$ has at most two orbits on $X$. If $G_{u,v,x}$ is transitive on $X$, then $X=\{x\}$,  
  so (i) and (ii) hold. Otherwise, $G_{u,v,x}$ and $G_{u,v,w}$ have the same orbits on $X$, namely $\{x\}$ and $X\setminus \{x\}$, so (ii) holds. Further, since $G_{u,v,x}$ is transitive on $X\setminus \{x\}$, either   $x$ has no neighbours in $X\setminus \{x\}$, or $x$ is adjacent to every vertex in $X\setminus \{x\}$. Since $G_{u,v}$ acts transitively on $X$, it follows that (i) holds. 
\end{proof}    

The following property will enable us to state a useful consequence of Lemma~\ref{lemma:x}.

\begin{definition}
\label{defn:x}
 A  non-complete graph $\Gamma$ has the \textit{unique $x$ property} if for some  $u\in V\Gamma$, $v\in \Gamma(u)$ and $w\in \Gamma_2(u)\cap \Gamma(v)$, there exists a unique $x\in \Gamma(u)\cap \Gamma_2(v)$ such that $\Gamma(u)\cap\Gamma(v)\cap \Gamma(w)=\Gamma(u)\cap \Gamma(v)\cap \Gamma(x)$.
\end{definition}

\begin{lemma}
\label{lemma:xplus}
 Let $\Gamma$ be a graph in which   $\Gamma(u')\cap\Gamma_2(v')$ is neither edgeless nor complete for some $u'\in V\Gamma$ and $v'\in\Gamma(u')$. If $\Gamma$ has the unique $x$ property,  then $\Gamma$ is not $4$-CH.
\end{lemma}

\begin{proof}
Suppose that $\Gamma$ is $4$-CH and has the unique $x$ property. There exist $u\in V\Gamma$, $v\in \Gamma(u)$, $w\in \Gamma_2(u)\cap \Gamma(v)$ and a unique $x\in \Gamma(u)\cap \Gamma_2(v)$ such that $\Gamma(u)\cap\Gamma(v)\cap \Gamma(w)=\Gamma(u)\cap \Gamma(v)\cap \Gamma(x)$. By assumption, $\Gamma(u')\cap\Gamma_2(v')$ is neither edgeless nor complete for some $u'\in V\Gamma$ and $v'\in\Gamma(u')$, and since $\Gamma$ is $2$-CH, it follows that  $\Gamma(u)\cap \Gamma_2(v)$ is neither edgeless nor complete. If $g\in \Aut(\Gamma)_{u,v,w}$, then $x^g\in \Gamma(u)\cap\Gamma_2(v)$ and $\Gamma(u)\cap\Gamma(v)\cap\Gamma(w)=\Gamma(u)\cap\Gamma(v)\cap\Gamma(x^g)$,  so $x^g=x$, but this is impossible by Lemma~\ref{lemma:x}.
\end{proof}

When the local graph of $\Gamma$ has diameter $2$, Lemma~\ref{lemma:xplus} says the following: if the $\mu$-graph in $\Gamma(v)$ of some $u,w\in \Gamma(v)$  is also the $\mu$-graph in $\Gamma(u)$ of $v$ and exactly one other vertex $x\in \Gamma(u)$,  then either $\Gamma(u)\cap\Gamma_2(v)$ is edgeless or complete, or $\Gamma$ is not $4$-CH.

One of the immediate consequences of the definition of $2$-homogeneity is that every connected non-complete $2$-homogeneous graph has diameter $2$ (see Lemma \ref{lemma:2homsimple}). However, this need not be the case for $4$-CH graphs: using Lemma~\ref{lemma:detkCH}, it is routine to verify that   the $n$-cube is a $4$-CH graph with diameter~$n$. We now establish some sufficient conditions for a connected $4$-CH graph to have diameter $2$. 

\begin{lemma}
\label{lemma:1fordiam2}
Let $\Gamma$ be a connected   $4$-CH graph. If there exists $u\in V\Gamma$, $v\in \Gamma(u)$ and $w\in \Gamma_2(u)\cap \Gamma(v)$ such that $\Gamma(u)\cap\Gamma_2(v)\cap\Gamma(w)\neq \varnothing$ and the graph induced by $\Gamma(u)\cap\Gamma_2(v)$ is connected, then $\diam(\Gamma)=2$.
\end{lemma}

\begin{proof}
Let $G:=\Aut(\Gamma)$.
Suppose for a contradiction that $\Gamma$ contains a $3$-geodesic. By assumption, there exists $x\in \Gamma(u)\cap\Gamma_2(v)\cap\Gamma(w)$. Now $G$ acts transitively on the set of $2$-geodesics in $\Gamma$, and $(v,w,x)$ is a $2$-geodesic, so 
 there exists $y\in V\Gamma$ such that $(v,w,x,y)$ is a $3$-geodesic. Suppose that $z$ is a neighbour of $x$ in $ \Gamma(u)\cap\Gamma_2(v)$. If $z$ is not adjacent to $w$, then the subgraphs induced by $\{v,w,x,y\}$ and $\{v,w,x,z\}$ are isomorphic, so   there exists $g\in G_{v,w,x}$ such that $y^g=z$, but $d_\Gamma(v,y)=3$ while $d_\Gamma(v,z)=2$, a contradiction. Since the graph induced by  $\Gamma(u)\cap\Gamma_2(v)$ is connected, it follows that $w$ is adjacent to every vertex in $\Gamma(u)\cap \Gamma_2(v)$. Thus $\Gamma(u)\cap \Gamma_2(v)\subseteq \Gamma(w)\cap \Gamma_2(v)$. Since $(u,v,w)$ is a $2$-geodesic,  there exists $g\in G_v$ such that $u^g=w$ and $w^g=u$, and it follows that $\Gamma(u)\cap \Gamma_2(v)=\Gamma(w)\cap \Gamma_2(v)$. Further, there exists $y'\in V\Gamma$ such that $(u,v,w,y')$ is a $3$-geodesic, but then $y'\in (\Gamma(w)\cap\Gamma_2(v))\setminus \Gamma(u)$, a contradiction. 
\end{proof}

 Note that a graph $\Gamma$ has an induced $4$-cycle if and only if there exists $u\in V\Gamma$, $v\in \Gamma(u)$ and $w\in \Gamma_2(u)\cap \Gamma(v)$ such that $\Gamma(u)\cap\Gamma_2(v)\cap\Gamma(w)\neq \varnothing$ (as in the statement of Lemma~\ref{lemma:1fordiam2}), and this occurs precisely when  some $\mu$-graph of $\Gamma$ is not complete.

\begin{lemma}
\label{lemma:2fordiam2}
Let $\Gamma$ be a    $3$-CH graph where $\Gamma(u)\cap\Gamma_2(v)\cap\Gamma(w)=\varnothing$ for some $u\in V\Gamma$, $v\in \Gamma(u)$ and $w\in \Gamma_2(u)\cap\Gamma(v)$. Then every $\mu$-graph of $\Gamma(u)$ is complete.
\end{lemma}

\begin{proof}
 By assumption, $v$ is adjacent to every vertex in $\Gamma(u)\cap\Gamma(w)\setminus \{v\}$, but $\Aut(\Gamma)_{u,w}$ is transitive on $\Gamma(u)\cap\Gamma(w)$, so the graph induced by $\Gamma(u)\cap\Gamma(w)$ is complete. It follows that every $\mu$-graph of $\Gamma$ is complete. Let $\Sigma$ be the graph induced by $\Gamma(u)$, and let $x\in V\Sigma$ and $y\in \Sigma_2(x)$. Now $\Gamma(x)\cap\Gamma(y)$ induces a complete graph, so $\Sigma(x)\cap\Sigma(y)$ does as well.
\end{proof}

\begin{lemma}
\label{lemma:diam2}
Let $\Gamma$ be a connected $4$-CH locally $\Sigma$ graph, where $\Sigma$ is a connected graph for which there exists $v\in V\Sigma$ and $x\in \Sigma_2(v)$ such that the graph induced by $\Sigma_2(v)$ is connected and the graph induced by $\Sigma(v)\cap\Sigma(x)$ is not complete. Then $\diam(\Gamma)=2$.
\end{lemma}

\begin{proof}
Let $u\in V\Gamma$, and view $v$ as a vertex in  $\Gamma(u)$. There exists $w\in \Gamma_2(u)\cap\Gamma(v)$. Some $\mu$-graph of $\Sigma$ is not complete by assumption, so Lemma~\ref{lemma:2fordiam2} implies that  $\Gamma(u)\cap\Gamma_2(v)\cap\Gamma(w)\neq \varnothing$. Since $\Sigma$ has diameter $2$ by Lemma \ref{lemma:3CH}, the graphs induced by $\Sigma_2(v)$ and $\Gamma(u)\cap\Gamma_2(v)$ are isomorphic, so $\Gamma(u)\cap\Gamma_2(v)$ induces a connected graph. Thus $\diam(\Gamma)=2$ by Lemma~\ref{lemma:1fordiam2}.
\end{proof}

\begin{remark}
Lemma~\ref{lemma:diam2} does not hold for $3$-CH graphs: using Lemma~\ref{lemma:detkCH}, it is routine to verify that the halved $n$-cube  is a $3$-CH graph with diameter $\lfloor n/2\rfloor$ whose local graph   satisfies the conditions of Lemma~\ref{lemma:diam2} for $n\geq 4$.
\end{remark}

If $\Gamma$ is a locally finite connected $4$-CH graph whose local graph $\Sigma$ satisfies the conditions of Lemma~\ref{lemma:diam2}, then $\Gamma$ is a finite  $2$-homogeneous graph and therefore known (see \S\ref{s:results}). In fact, it turns out that  most  finite $3$-homogeneous graphs satisfy the conditions on $\Sigma$ in Lemma~\ref{lemma:diam2}, so we could prove Theorem~\ref{thm:connected} using a case-by-case analysis of the finite  $2$-homogeneous graphs. However,  there are many more families of  finite $2$-homogeneous graphs than finite $3$-homogeneous graphs (see \cite{LieSax1986}), and 
we prefer   the more direct and elementary  approach provided by Lemma~\ref{lemma:x}.

Next we have some results that will be useful in the locally disconnected case when  $c_2=1$. Note that in both of these results, we are not necessarily assuming that $\diam(\Gamma)$ is finite, but since $k$ is an integer, the parameter $m$ is also an integer.
 
 \begin{lemma}
 \label{lemma:ctrick}
 Let $\Gamma$ be a connected $k$-CH graph  that is locally $(t+1)\cdot K_s$ for positive integers $t$ and $s$.  Let $m:=\min(\diam(\Gamma),k-1)$ and suppose that  $c_i(\Gamma)=1$ for some $1<  i< m$ (so $k\geq 4$). Let $u\in V\Gamma$, $v\in \Gamma_i(u)$ and $y\in \Gamma_{i-1}(u)\cap \Gamma(v)$. Then $\Gamma_i(u)\cap \Gamma(v)=\Gamma(v)\cap \Gamma(y)$.
\end{lemma}
 
 \begin{proof}
  There exists a unique clique $\mathcal{C}$ of size $s+1$ containing $v$ and $y$, and $\Gamma_{i-1}(u)\cap\Gamma(v)=\{y\}$ since $c_i=1$, so $\Gamma(v)\cap\Gamma(y)=\mathcal{C}\setminus \{y,v\}\subseteq \Gamma_i(u)\cap\Gamma(v)$.  There exists a path $(u_0,\ldots,u_i)$ where $u_0=u$, $u_{i-1}=y$ and $u_i=v$, and since $\diam(\Gamma)>i$ and $\Gamma$ is $(i+1)$-CH, there exists $w\in\Gamma_{i+1}(u)\cap\Gamma(v)$. If there exists $x\in \Gamma_i(u)\cap\Gamma(v)\setminus \Gamma(y)$, then  $\{u_0,\ldots,u_i,x\}$ and $\{u_0,\dots,u_i,w\}$ induce isomorphic subgraphs of $\Gamma$ with order $i+2$, and $i+2\leq k$, so there exists $g\in \Aut(\Gamma)_{u_0,\ldots,u_i}$ such that $x^g=w$,   but $d_\Gamma(u,x)=i$ while $d_\Gamma(u,w)=i+1$,   
  a contradiction. Thus $\Gamma_i(u)\cap \Gamma(v)=\Gamma(v)\cap \Gamma(y)$. 
  \end{proof}

\begin{lemma}
\label{lemma:c2=1}
Let $\Gamma$ be a connected $k$-CH graph  that is locally $(t+1)\cdot K_s$ for positive integers $t$ and $s$ where $k\geq 3$ and $c_2(\Gamma)=1$. Let $m:=\min(\diam(\Gamma),k-1)$. Then the following hold.
\begin{itemize}
\item[(i)]  $c_i(\Gamma)=1$ and $a_i(\Gamma)=s-1$ for $1\leq i<m$.
\item[(ii)] For $u,v\in V\Gamma$ such that $d_\Gamma(u,v)=m$, the set $\Gamma_{m-1}(u)\cap\Gamma(v)$ induces $c_m(\Gamma)\cdot K_1$. 
\item[(iii)] If $\diam(\Gamma)\leq k-1$, then $\Gamma$ is distance-transitive. 
\item[(iv)] If $\diam(\Gamma)<k-1$, then either  $c_m(\Gamma)=1$, or $s=1$ and $c_m(\Gamma)=t+1$.
\item[(v)] If   $3\leq \diam(\Gamma)<k-2$, then $s=1$, and either $\Gamma\simeq C_{2m+1}$ or  $c_m(\Gamma)=t+1$.
\end{itemize}
\end{lemma}

\begin{proof}
Let $G:=\Aut(\Gamma)$. Since $\Gamma$ is $(m+1)$-CH, the parameters $a_i$ and $c_i$ are defined for $1\leq i\leq m$, and (iii) holds. 
By assumption,  $a_1=s-1$ and $c_1=c_2=1$. 

First we prove that (i) holds. Suppose for a contradiction that $c_{i-1}=1$ for some $2<i<m$ (so $k\geq 5$) but $c_i\neq 1$.   Then there exists $u\in V\Gamma$, $v\in\Gamma_i(u)$,   $w\in\Gamma_{i+1}(u)\cap\Gamma(v)$, distinct $x,y\in \Gamma_{i-1}(u)\cap\Gamma(v)$, and a path $(u_0,\ldots,u_i)$ where $u_0=u$, $u_{i-1}=y$ and $u_i=v$. Since $c_2=1$, $x$ is not adjacent to $u_{i-2}$. Further, $x$ is not adjacent to $y$ since  $ \Gamma_{i-1}(u)\cap\Gamma(y)=\Gamma(y)\cap\Gamma(u_{i-2})$ by Lemma~\ref{lemma:ctrick}. Hence    $\{u_0,\ldots,u_i,x\}$ and $\{u_0,\dots,u_i,w\}$ induce   path graphs with order $i+2$, and $i+2\leq k$, so there exists $g\in G_{u_0,\ldots,u_i}$ such that $x^g=w$,  a contradiction. Thus $c_i=1$ for $1\leq i<m$. In particular, $a_i=s-1$ for $1< i<m$ by Lemma~\ref{lemma:ctrick}. Since $a_1=s-1$,  (i) holds. 

For the remainder of the proof, let $(u_0,u_1,\ldots,u_m)$ be a geodesic in $\Gamma$. Let $u:=u_0$, $v:=u_m$ and $y:=u_{m-1}$. Note that $y\in \Gamma_{m-1}(u)\cap \Gamma(v)$.

First we claim that $y$ has no neighbours in $\Gamma_{m-1}(u)\cap \Gamma(v)$, in which case (ii) holds since $|\Gamma_{m-1}(u)\cap \Gamma(v)|=c_m$. If $m=2$, then the claim is trivial since $c_2=1$. Suppose instead that  $m\geq 3$. By (i), $c_{m-1}=1$, so $\Gamma_{m-1}(u)\cap \Gamma(y)=\Gamma(y)\cap \Gamma(u_{m-2})$  by Lemma~\ref{lemma:ctrick}.   Suppose for a contradiction that $x$ is a neighbour of $y$ in $\Gamma_{m-1}(u)\cap \Gamma(v)$. Now   $x$ is adjacent to $u_{m-2}$, but then  $u_{m-2}$ and $v$ are vertices at distance   $2$ in $\Gamma$  with  common neighbours  $y$ and $x$,  contradicting our assumption that $c_2=1$. Thus the claim holds.

Next we prove that (iv) holds.  Suppose that $\diam(\Gamma)<k-1$ and $c_m>1$. Now there exists $x\in \Gamma_{m-1}(u)\cap\Gamma(v)$ such that $x\neq y$. Observe that $\{u_0,\ldots,u_m,x\}$ induces a  path graph by (ii) and the fact that $c_2=1$.  If there exists $w\in (\Gamma_m(u)\cap \Gamma(v))\setminus \Gamma(y)$, then  $\{u_0,\dots,u_m,w\}$ also induces a path graph, so there exists $g\in G_{u_0,\ldots,u_m}$ such that $x^g=w$,  a contradiction.  Thus $\Gamma_m(u)\cap \Gamma(v)\subseteq \Gamma(y)$.  Similarly, $\Gamma_m(u)\cap \Gamma(v) \subseteq \Gamma(x)$, but $\Gamma(y)\cap \Gamma(x)=\{v\}$ since  $d_\Gamma(x,y)=2$,  so   $\Gamma_m(u)\cap \Gamma(v)=\varnothing$.  It then follows from (ii) that  $s=1$ and $c_m=t+1$, so (iv) holds.

Finally, we prove that (v) holds.
Suppose that  $3\leq \diam(\Gamma)<k-2$. If $s=1$ and $c_m=t+1$, then (v) holds, so we may assume otherwise. Then $c_m=1$ by (iv).  Now there exists $x\in (\Gamma(v)\cap\Gamma_{m}(u))\setminus \Gamma(y)$. Let $z$ be the unique vertex in $\Gamma_{m-1}(u)\cap\Gamma(x)$. Since $c_m=1$ and $c_2=1$, the set $\{y,v,x,z\}$ induces a path graph. By Lemma~\ref{lemma:ctrick}, $\Gamma_2(v)\cap\Gamma(u_{m-2})=\Gamma(u_{m-2})\cap\Gamma(y)$, so $z$ is not adjacent to $u_{m-2}$. Thus $\{u_0,\ldots,u_m,x,z\}$ induces a path graph. If $s\geq 2$, then there exists $w\in \Gamma(x)\cap\Gamma(z)$, and $w\in \Gamma_m(u)$, but $w$ is not adjacent to $v$ or $y$ since $c_2=1$, so $\{u_0,\ldots,u_m,x,w\}$ also induces a path graph, in which case there exists $g\in G_{u_0,\ldots,u_m,x}$ such that $z^g=w$,  a contradiction. Thus $s=1$. If $t\geq 2$, then there exists $w\in \Gamma(x)\setminus \{v,z\}$. Again  $w\in \Gamma_m(u)$ and $w$ is not adjacent to $v$ or $y$, so  $\{u_0,\ldots,u_m,x,w\}$ induces a path graph, a contradiction as above. Hence $t=1$, so  $\Gamma\simeq C_n$ for some $n$. Since $\Gamma$ has diameter $m$ and $c_i=1$ for $1\leq i\leq m$, it follows  that $n=2m+1$.
\end{proof}

\begin{remark}
One consequence of the classification of the locally finite CH graphs~\cite{Gar1978,Eno1981} (see Theorem~\ref{thm:CH}) is that the only locally finite, connected,  locally disconnected CH graphs with girth $3$ and $c_2=1$ are halved graphs of the biregular tree $T_{t+1,s+1}$. In particular, no such graph is finite. For graphs with diameter at least $3$, these facts can be deduced directly from Lemma~\ref{lemma:c2=1}.
\end{remark}

\begin{lemma}
\label{lemma:cyclebig}
Let $\Gamma$ be a connected $k$-CH graph  that is locally $(t+1)\cdot K_s$  where $t\geq 1$, $s\geq 2$,  $c_2(\Gamma)=1$ and $\diam(\Gamma)\geq 3$. If  some induced subgraph of $\Gamma$ is isomorphic to $C_r$ for some $r> 3$, then  $r\geq k+2$.
\end{lemma}

\begin{proof}
 Let $\Delta$ be an induced subgraph of $\Gamma$ that is isomorphic to $C_r$ where $r> 3$. Since $c_2(\Gamma)=1$, $r\geq 5$. If $k\leq 3$, then $r\geq k+2$, as desired, so we assume that $k\geq 4$.
 
Write $r=2n+1$ or $2n+2$ where $n\geq 2$.  Label the vertices of $V\Delta$ as follows: choose $u\in V\Delta$, and write $\Delta_i(u)=\{x_i,y_i\}$ for $1\leq i\leq n$, where $x_i$ is adjacent to $x_{i+1}$  for $1\leq i\leq n-1$, and therefore $y_i$ is adjacent to $y_{i+1}$  for $1\leq i\leq n-1$. If $r=2n+1$, then $x_n$ is adjacent to $y_n$, while if  $r=2n+2$, then   $\Delta_{n+1}(u)=\{z\}$ where $z$ is adjacent to $x_n$ and $y_n$. 

Let $m:=\min(\diam(\Gamma),k-1)$, and recall  that $m\geq 3$. Since $s\geq 2$ and $\diam(\Gamma)\geq 3$, Lemma~\ref{lemma:c2=1}(iv) and (v) imply that either $m\geq k-1$, or $m=k-2$ and  $c_m(\Gamma)=1$ (in which case $k\geq 5$). In particular,  
$2m+1\geq k+2$.  Hence if $n\geq m$, then $r\geq 2m+1\geq k+2$, as desired, so we assume instead that $n<m$.

 Observe that $x_2,y_2\in \Gamma_2(u)$ since $V\Delta$ induces $C_r$ where $r>3$.  Suppose that  $x_i\in \Gamma_i(u)$ for some $2\leq i<n$. Now $x_{i-1}\in \Gamma_{i-1}(u)\cap \Gamma(x_i)$. Further, $c_i(\Gamma)=1$ by Lemma~\ref{lemma:c2=1}(i),  so $x_{i+1}\notin \Gamma_{i-1}(u)$, and if  $x_{i+1}\in \Gamma_i(u)$, then  $x_{i+1}\in \Gamma_i(u)\cap\Gamma(x_i)=\Gamma(x_i)\cap \Gamma(x_{i-1})$ by Lemma~\ref{lemma:ctrick}, a contradiction since $V\Delta$ induces $C_r$. Thus $x_{i+1}\in \Gamma_{i+1}(u)$. It follows that  $x_j,y_j\in \Gamma_j(u)$ for $1\leq j\leq n$. 
 By Lemma~\ref{lemma:c2=1}(i),  $c_n(\Gamma)=1$. If $r=2n+1$, then $y_n$ is adjacent to $x_n$, but then $y_n$ is adjacent to $x_{n-1}$ by Lemma~\ref{lemma:ctrick}, a contradiction. Thus $r=2n+2$, and by a similar argument, $z\notin \Gamma_n(u)$. Since $c_n(\Gamma)=1$, it follows that $z\in \Gamma_{n+1}(u)$.
 In particular, $c_{n+1}(\Gamma)\geq 2$. If $n+1<m$, then 
$c_{n+1}(\Gamma)=1$ by Lemma~\ref{lemma:c2=1}(i), a contradiction. Thus $m=n+1$. We saw above that either $m\geq k-1$, or $m=k-2$ and  $c_m(\Gamma)=1$. Since $c_m(\Gamma)\neq 1$, we conclude that $m\geq k-1$. Thus $r=2n+2=2m\geq 2(k-1)\geq k+2$, as desired.
\end{proof}

\section{Families of graphs with the unique $x$ property}
\label{s:uniquex}

In this section, we give five different sets of combinatorial conditions on the local structure of a graph $\Gamma$ which guarantee that $\Gamma$ has the unique $x$ property  (see   Definition~\ref{defn:x}). These will be used to prove Theorem~\ref{thm:connected} in conjunction with Lemma~\ref{lemma:xplus}. Recall that, by our definition, all strongly regular graphs are finite non-complete graphs. In particular, a strongly regular graph with $\mu>0$ is connected with diameter  $2$.

\begin{lemma}
\label{lemma:2K1}
 Let $\Gamma$ be a  locally $\Sigma$ graph where $\Sigma$ is a strongly regular graph with $\mu(\Sigma)>0$ in which every $\mu$-graph is  $2\cdot K_1$. Then $\Gamma$ has the unique $x$ property.
\end{lemma}

\begin{proof}
Let $u\in V\Gamma$, $v\in\Gamma(u)$ and $w\in \Gamma_2(u)\cap \Gamma(v)$.  Since the graph induced by $\Gamma(v)$ has diameter $2$, it follows that the graph $\Delta$ induced by $\Gamma(u)\cap\Gamma(v)\cap\Gamma(w)$ is a $\mu$-graph of $\Gamma(v)$ and is therefore isomorphic to $2\cdot K_1$. Let $\Sigma$ denote the graph induced by $\Gamma(u)$.

Let $V\Delta=\{y,z\}$, and note that $d_\Sigma(y,z)=2$. Now $v\in \Sigma(y)\cap\Sigma(z)$ and $\Sigma(y)\cap\Sigma(z)\simeq 2\cdot K_1$, so $\Sigma(y)\cap\Sigma(z)=\{v,x\}$ for some $x\in \Sigma_2(v)$. Further, $\Sigma(v)\cap\Sigma(x)$ is also isomorphic to $2\cdot K_1$, so  $\Sigma(v)\cap\Sigma(x)=\{y,z\}$. Thus $x\in \Gamma(u)\cap\Gamma_2(v)$ and $ V\Delta=\{y,z\}=\Gamma(u)\cap\Gamma(v)\cap\Gamma(x)$, and $x$ is the unique such vertex.
\end{proof}

\begin{lemma}
\label{lemma:k2=1}
 Let $\Gamma$, $\Sigma$ and $\Delta$ be graphs   with the following properties.
\begin{itemize}
\item[(i)] $\Gamma$ is locally $\Sigma$, and every $\mu$-graph of $\Sigma$ is isomorphic to $\Delta$.
 \item[(ii)] $\Sigma$ is strongly regular with $\mu(\Sigma)>0$.
\item[(iii)] $\Delta$ is regular with diameter $2$ and  $k_2(\Delta)=1$.
\end{itemize}
Then $\Gamma$ has the unique $x$ property.
\end{lemma}

\begin{proof}
 Let $u\in V\Gamma$, $v\in\Gamma(u)$ and $w\in \Gamma_2(u)\cap \Gamma(v)$.  The graph $\Gamma(v)$ has diameter $2$, so $\Gamma(u)\cap\Gamma(v)\cap\Gamma(w)$ is a $\mu$-graph of $\Gamma(v)$ and is therefore isomorphic to $\Delta$. For the remainder of the proof, we  write $\Sigma$ for the graph induced by $\Gamma(u)$ and $\Delta$ for the graph induced by $\Gamma(u)\cap\Gamma(v)\cap\Gamma(w)$.

Let $y\in V\Delta$. Since $k_2(\Delta)=1$, there exists a unique $z\in \Delta_2(y)$. Now $\Sigma(y)\cap\Sigma(z)\simeq  \Delta$  and $v\in \Sigma(y)\cap\Sigma(z)$, so there is a unique $x$ at distance $2$ from $v$ in $ \Sigma(y)\cap\Sigma(z)$. Then $\Sigma(v)\cap\Sigma(x)\simeq \Delta$, and $y$ and  $z$ are at distance $2$ in $\Sigma(v)\cap\Sigma(x)$. Note that $\Delta$ is finite since $\Sigma$ is finite by definition. Since $\Delta$ is regular,
$$|\Sigma(v)\cap\Sigma(y)\cap\Sigma(x)|=|\Delta(y)|=|\Sigma(v)\cap\Sigma(y)\cap\Sigma(z)|.$$ 
Further, since $k_2(\Delta)=1$, both $\Sigma(v)\cap\Sigma(y)\cap\Sigma(x)$ and $\Delta(y)$ 
are subsets of $\Sigma(v)\cap\Sigma(y)\cap\Sigma(z)$. Thus $\Delta(y)=\Sigma(v)\cap\Sigma(y)\cap\Sigma(x)$. It follows that   
$$V\Delta=\{y,z\}\cup\Delta(y)=\Sigma(v)\cap\Sigma(x)=\Gamma(u)\cap\Gamma(v)\cap\Gamma(x).$$ 
It remains to show that $x$ is the unique vertex in $\Gamma(u)\cap\Gamma_2(v)$ such that $V\Delta=\Gamma(u)\cap\Gamma(v)\cap\Gamma(x)$. If $x'$ is such a vertex, then $x'\in\Sigma_2(v)\cap\Sigma(y)\cap\Sigma(z)=\{x\}$.
\end{proof}

\begin{lemma}
\label{lemma:polar}
 Let $\Gamma$, $\Sigma$ and $\Delta$ be  graphs  with the following properties.
\begin{itemize}
\item[(i)] $\Gamma$ is locally $\Sigma$, and every $\mu$-graph of $\Sigma$ is isomorphic to $\Delta$.
 \item[(ii)] $\Sigma$ is strongly regular with  $\lambda(\Sigma)=\mu(\Sigma)-2$.
\item[(iii)] $\Delta$ has diameter $3$ with $k_1(\Delta)=k_2(\Delta)$ and $k_3(\Delta)=1$.
\item[(iv)] For any distinct non-adjacent $y,z\in V\Sigma$, if $v$ and $x$ are vertices at distance $3$ in the graph induced by $\Sigma(y)\cap\Sigma(z)$, then $y$ and $z$ are at distance $3$ in the graph induced by $\Sigma(v)\cap\Sigma(x)$.
\end{itemize}
Then $\Gamma$ has the unique $x$ property.
\end{lemma}

\begin{proof}
Let $u\in V\Gamma$, $v\in\Gamma(u)$ and $w\in \Gamma_2(u)\cap \Gamma(v)$.  The graph $\Gamma(v)$ has diameter $2$, so $\Gamma(u)\cap\Gamma(v)\cap\Gamma(w)$ is a $\mu$-graph of $\Gamma(v)$ and is therefore isomorphic to $ \Delta$. For the remainder of the proof, we write $\Sigma$ for the graph induced by $\Gamma(u)$ and $\Delta$ for the graph induced by $\Gamma(u)\cap\Gamma(v)\cap\Gamma(w)$.

Let $y\in V\Delta$. Since $k_3(\Delta)=1$, there exists a unique $z\in \Delta_3(y)$.  Now $\Sigma(y)\cap\Sigma(z)\simeq  \Delta$  and $v\in \Sigma(y)\cap\Sigma(z)$, so there is a unique $x$ at distance $3$ from $v$ in the graph induced by $ \Sigma(y)\cap\Sigma(z)$. Then $\Sigma(v)\cap\Sigma(x)\simeq \Delta$, and $y$ and  $z$ are at distance $3$ in the graph induced by $\Sigma(v)\cap\Sigma(x)$ by assumption.

We claim that $\Delta(y)=\Sigma(v)\cap\Sigma(y)\cap\Sigma(x)$. Since $v$ and $y$ are adjacent,  $|\Sigma(v)\cap\Sigma(y)|=\lambda(\Sigma)$. Now $2+2|\Delta(y)|=|V\Delta|=\mu(\Sigma)$ 
and $\lambda(\Sigma)=\mu(\Sigma)-2$, so  $|\Delta(y)|=\lambda(\Sigma)/2$. Since $\Delta$ is regular, $$|\Sigma(v)\cap\Sigma(y)\cap\Sigma(x)|=\lambda(\Sigma)/2=|\Sigma(v)\cap\Sigma(y)\cap\Sigma(z)|.$$
Since $\Sigma$ has diameter $2$, any vertex in $\Sigma(v)\cap \Sigma(y)$ lies in $\Sigma(z)$ or $\Sigma_2(z)$. Thus
 $$|\Sigma(v)\cap\Sigma(y)\cap\Sigma_2(z)|=\lambda(\Sigma)/2.$$
Now $\Sigma(v)\cap\Sigma(y)\cap\Sigma(x)$ and $\Delta(y)$ are both subsets of $\Sigma(v)\cap\Sigma(y)\cap\Sigma_2(z)$, so the claim follows.

By exchanging the roles of $y$ and $z$ in the above proof, we also obtain $\Delta(z)=\Sigma(v)\cap\Sigma(z)\cap\Sigma(x)$. Thus 
$$V\Delta=\{y,z\}\cup\Delta(y)\cup\Delta(z)=\Sigma(v)\cap\Sigma(x)=\Gamma(u)\cap\Gamma(v)\cap\Gamma(x).$$ 
It remains to show that $x$ is the unique vertex in $\Gamma(u)\cap\Gamma_2(v)$ such that $V\Delta=\Gamma(u)\cap\Gamma(v)\cap\Gamma(x)$. If $x'$ is another such vertex, then $V\Delta\subseteq \Sigma(x')$, so $x'\in \Sigma(y)\cap\Sigma(z)\setminus (\{x,v\}\cup \Sigma(v))$, in which case $x'\in \Sigma(x)$, but then $|\Sigma(x)\cap\Sigma(x')|=\lambda(\Sigma)<\mu(\Sigma)=|V\Delta|$, contradicting $V\Delta\subseteq \Sigma(x)\cap\Sigma(x')$.
\end{proof}

\begin{lemma}
\label{lemma:polarcomp}
 Let $\Gamma$, $\Sigma$ and $\Delta$  be  graphs  with the following properties.
\begin{itemize}
\item[(i)] $\Gamma$ is locally $\Sigma$, and every $\mu$-graph of $\Sigma$ is isomorphic to $\Delta$. 
\item[(ii)] $\Sigma$ is strongly regular with $\mu(\Sigma)>0$. 
\item[(iii)] $\Delta$ is regular and for any $y\in V\Delta$, there exists a unique $z\in \Delta_2(y)$ such that $\Delta(y)=\Delta(z)$.
\item[(iv)] $\Sigma$ has valency $\mu(\Sigma)+2\lambda(\Sigma)-2k_1(\Delta)$.
 \end{itemize}
Then $\Gamma$ has the unique $x$ property.
\end{lemma}

\begin{proof}
Let $u\in V\Gamma$, $v\in\Gamma(u)$ and $w\in \Gamma_2(u)\cap \Gamma(v)$.  The graph $\Gamma(v)$ has diameter $2$, so $\Gamma(u)\cap\Gamma(v)\cap\Gamma(w)$ is a $\mu$-graph of $\Gamma(v)$ and is therefore isomorphic to $ \Delta$. For the remainder of the proof, we write $\Sigma$ for the graph induced by $\Gamma(u)$; $\Delta$ for the graph induced by $\Gamma(u)\cap\Gamma(v)\cap\Gamma(w)$; and $\Pi$ for the graph induced by $\Sigma(v)$.

Let $y\in V\Delta$. By assumption there exists a unique $z\in \Delta_2(y)$ such that $\Delta(y)=\Delta(z)$. Now the graph $\Lambda$ induced by $\Sigma(y)\cap\Sigma(z)$ is isomorphic to $\Delta$, and $v\in \Sigma(y)\cap\Sigma(z)$, so there exists a unique $x\in \Lambda_2(v)$ such that $\Lambda(v)=\Lambda(x)$. Then the graph $\Theta$ induced by $\Sigma(v)\cap \Sigma(x)$ is isomorphic to $ \Delta$, and $\Delta(y)\subseteq \Lambda(v)\subseteq \Theta(y)\cap \Theta(z)$, so $\Delta(y)=\Lambda(v)=\Theta(y)=\Theta(z)$. In particular, $\Pi(y)\cap \Pi(z)=\Sigma(v)\cap \Sigma(y)\cap\Sigma(z)=\Lambda(v)=\Delta(y)$. Let $X:=V\Theta\setminus (\{y,z\}\cup \Theta(y))$, $Y:=V\Delta\setminus (\{y,z\}\cup \Delta(y))$ and $ Z:=V\Pi\setminus (\{y,z\}\cup \Pi(y)\cup\Pi(z))$. Note that $X,Y\subseteq Z$.  Now
$|X|=\mu(\Sigma)-(k_1(\Delta)+2)=|Y|$
and 
$|Z|=|V\Pi|-(2\lambda(\Sigma)-k_1(\Delta)+2)$,
but $|V\Pi|=\mu(\Sigma)+2\lambda(\Sigma)-2k_1(\Delta)$ by assumption, so $|X|=|Y|=|Z|$. Thus $X=Z=Y$, so $V\Delta=V\Theta=\Gamma(u)\cap\Gamma(v)\cap\Gamma(x)$. It remains to show that $x$ is the unique vertex in $\Gamma(u)\cap\Gamma_2(v)$ such that $V\Delta=\Gamma(u)\cap\Gamma(v)\cap\Gamma(x)$. If $x'$ is such a vertex, then $V\Delta\subseteq \Sigma(x')$, so $x'\in \Lambda_2(v)$ and $\Lambda(v)=\Delta(y)\subseteq \Lambda(x')$. Since $\Lambda$ is regular,   $\Lambda(v)=\Lambda(x')$, but $x$ is the unique such vertex by assumption. Thus $x=x'$.
\end{proof}

We now wish to show that the unique $x$ property holds in graphs that are locally $\Sigma$, where the complement graph $\overline{\Sigma}$ is the point graph of a thick generalised quadrangle (see \S\ref{s:defn}). To simplify the proof, we first establish some elementary properties of $\Sigma$. 
 
\begin{lemma}
\label{lemma:compGQ}
Let $\Sigma$ be the complement of the point graph of a thick generalised quadrangle of order $(s,t)$. 
Then $\Sigma$ is strongly regular with parameters $((st+1)(s+1),s^2t,s^2t-st-s+t,(s-1)st)$ and $\Sigma_2(v)$ induces $K_{(t+1)[s]}$ for all $v\in V\Sigma$. Further, the following hold.
\begin{itemize}
\item[(i)] Any distinct non-adjacent $v,x\in V\Sigma$ lie in a unique independent set $I(v,x)$ of $\Sigma$ with size $s+1$, and if $I$ is an independent set of $\Sigma$ containing $v$ and $x$, then $I\subseteq I(v,x)$. 
\item[(ii)] Let $I$ and $J$ be independent sets of $\Sigma$ with size $s$. Let $c(I):=\bigcap_{a\in I}\Sigma(a)$, and let $x_I$ denote the unique vertex in $V\Sigma$ for which $I\cup \{x_I\}$ is independent. Then $\Sigma_2(x_I)=c(I)\cup I$. If $I\cap J\neq \varnothing$, then either $I\cup \{x_I\}=J\cup \{x_J\}$, or $c(I)\cap c(J)\neq \varnothing$.     
\item[(iii)] Let $\Delta$ be a $\mu$-graph of $\Sigma$. There exists a partition $\mathcal{P}$ of $V\Delta$ into $s-1$ parts such that each  $P\in\mathcal{P}$ induces  $K_{t[s]}$. If $I$ is a maximal independent set of  $P\in\mathcal{P}$, then $\bigcap_{a\in I}\Delta(a)\subseteq P$.
\end{itemize}
\end{lemma}

\begin{proof}
 The point graph of a thick generalised quadrangle of order $(s,t)$ is a strongly regular graph with parameters $((st+1)(s+1),s(t+1),s-1,t+1)$, so $\Sigma$ is strongly regular with parameters $((st+1)(s+1),s^2t,s^2t-st-s+t,(s-1)st)$. In particular, $\Sigma$ has diameter $2$ since $s\geq 2$. Let $\mathcal{L}$ be the set of maximal cliques of $\overline{\Sigma}$. Now $\mathcal{L}$ is also the set of maximal independent sets of $\Sigma$.  By the GQ Axiom,  we may view $\mathcal{L}$ as the line set of the generalised quadrangle, and $\Sigma_2(v)$ induces $K_{(t+1)[s]}$ for all $v\in V\Sigma$. It follows that (i) holds.

Let $I$ and $J$ be independent sets of $\Sigma$ with size $s$, and let $c(I)$ and $x_I$ be as defined in (ii). Now $I\subseteq \Sigma_2(x_I)$ and $\Sigma_2(x_I)\setminus I\subseteq c(I)$. If there exists $y\in c(I)\setminus \Sigma_2(x_I)$, then $y$ is adjacent in $\Sigma$ to every vertex in $I\cup \{x_I\}$, contradicting the GQ Axiom. Thus $\Sigma_2(x_I)=c(I)\cup I$. Suppose that $v\in I\cap J$ and $I\cup \{x_I\}\neq J\cup \{x_J\}$. Then  $\ell_I:=I\cup \{x_I\}$ and $\ell_J:=J\cup \{x_J\}$ are distinct lines that contain $v$ and therefore  intersect in $v$. There exists $\ell\in \mathcal{L}$ such that $x_J\in \ell$ and $\ell\neq \ell_J$. By the GQ Axiom, $x_I\notin \ell$, and there exists $m\in \mathcal{L}$ such that $x_I\in m$ and $m$ intersects $\ell$ at a vertex $y$, and  $y\in (\Sigma_2(x_I)\cap \Sigma_2(x_J))\setminus (I\cup J)=c(I)\cap c(J)$.
Thus (ii) holds.

Let $\Delta$ be a $\mu$-graph of $\Sigma$. Then $\Delta$ is the graph induced by $\Sigma(v)\cap \Sigma(x)$ for some $v\in V\Sigma$ and $x\in \Sigma_2(v)$. Now there exists $\ell\in\mathcal{L}$ such that $v,x\in \ell$. Let $Y:=\ell\setminus \{v,x\}$, and for $y\in Y$,  let $P_y$ be the set of $u\in V\Sigma\setminus \ell$ such that $u$ is collinear with $y$. By the GQ Axiom, each $P_y$ induces (in $\Sigma$) the graph $K_{t[s]}$, and  the set $\{P_y: y \in Y\}$ is a partition of $V\Delta$ with $s-1$ parts. Let $I$ be a maximal independent set of $P_y$ for some $y\in Y$. Now $I\cup \{y\}\in \mathcal{L}$, so by the GQ Axiom, each $u\in V\Delta\setminus P_y$ is not collinear with $y$ and is therefore collinear with some vertex in $I$. It follows that $\bigcap_{a\in I}\Delta(a)\subseteq P_y$, and (iii) holds.
\end{proof}

\begin{lemma}
\label{lemma:localGQ}
 Let $\Gamma$, $\Sigma$ and $\Delta$ be  graphs  with the following properties.
\begin{itemize}
\item[(i)] $\Gamma$ is locally $\Sigma$, and every $\mu$-graph of $\Sigma$ is isomorphic to $\Delta$.
\item[(ii)] $\Sigma$ is the complement of the  point graph of a thick generalised quadrangle.
\end{itemize}
Then $\Gamma$ has the unique $x$ property.
\end{lemma}

\begin{proof}
Let $u\in V\Gamma$, $v\in \Gamma(u)$ and $w\in \Gamma_2(u)\cap\Gamma(v)$. The graph $\Gamma(v)$ has diameter $2$ since the generalised quadrangle is thick, so $\Gamma(u)\cap\Gamma(v)\cap\Gamma(w)$ is a $\mu$-graph of $\Gamma(v)$ and is therefore isomorphic to $ \Delta$. For the remainder of the proof, we write $\Sigma$ for the graph induced by $\Gamma(u)$ and $\Delta$ for the graph induced by $\Gamma(u)\cap\Gamma(v)\cap\Gamma(w)$. 
 
 Since the generalised quadrangle is thick, it has order  $(s,t)$ for some $s\geq 2$ and $t\geq 2$.  By Lemma~\ref{lemma:compGQ}, the graph $\Sigma_2(y)$ induces $K_{(t+1)[s]}$ for all $y\in V\Sigma$, and  any distinct non-adjacent  $y,z\in V\Sigma$ lie in a unique independent set of $\Sigma$ with size $s+1$; this we denote by $I(y,z)$.  For  an independent set $I$ of $\Sigma$ with size $s$, let $c(I):=\bigcap_{a\in I}\Sigma(a)$, let $x_I$ denote the unique element of $V\Sigma$ for which $I\cup \{x_I\}$ is independent, and recall that $\Sigma_2(x_I)=c(I)\cup I$ by Lemma~\ref{lemma:compGQ}(ii). In particular, $c(I)$ induces $K_{t[s]}$.
By Lemma~\ref{lemma:compGQ}(iii), there exists a partition $\mathcal{P}$ of $V\Delta$ into $s-1$ parts such that each $P\in\mathcal{P}$ induces the graph $K_{t[s]}$. 

Let $P\in \mathcal{P}$. We claim that there exists an independent set $I(P)$ with size $s$ such that $v\in I(P)$ and $P=c(I(P))$. Choose a maximal independent set $J$ of $P$ and distinct $y,z\in J$, which exist since $s\geq 2$. Let $\Lambda$ be the graph induced by $\Sigma(y)\cap\Sigma(z)$, which is isomorphic to $\Delta$. Now $v\in V\Lambda$, so  $v$ lies in an independent set $I(P)$ of $\Lambda$ with size $s$ for which $X:=c(I(P))\cap V\Lambda$ induces $K_{(t-1)[s]}$ by Lemma~\ref{lemma:compGQ}(iii). Further,  $c(I(P))$ induces $ K_{t[s]}$, so $I:=c(I(P))\setminus V\Lambda$ is an independent set with size $s$. Since $y,z\in I$, we conclude that $I\subseteq I(y,z)$. Similarly, $J\subseteq I(y,z)$. If $I\neq J$, then  $I(y,z)=J\cup I$, but $x_J\in I(y,z)$ and $J\cup I\subseteq \Sigma(v)$, so $v\in \Sigma(x_J)\cap c(J)$, contradicting  $\Sigma_2(x_J)=c(J)\cup J$. Thus $I=J$. Now $J \cup X\cup I(P)=c((I(P))\cup I(P)=\Sigma_2(x_{I(P)})$. Since $\Sigma_2(x_{I(P)})$ induces $K_{(t+1)[s]}$, it follows that $X\subseteq c(J)$.  Now $X$ and $c(J)\cap P$ are subsets of $c(J)\cap \Sigma(v)$ with size $(t-1)s$. Since $v\in c(J)$ and $c(J)$ induces $ K_{t[s]}$, it follows that $X=c(J)\cap P$. Thus $c(I(P))=I\cup X=J\cup (c(J)\cap P)=P$, as desired.

If $I(P)\cup \{x_{I(P)}\}\neq I(Q)\cup \{x_{I(Q)}\}$ for some $P,Q\in \mathcal{P}$, then since $v\in I(P)\cap I(Q)$, Lemma~\ref{lemma:compGQ}(ii) implies that $c(I(P))\cap c(I(Q))\neq \varnothing$, but then $P\cap Q\neq \varnothing$, so $P=Q$, a contradiction. Thus $I(P)\cup \{x_{I(P)}\}= I(Q)\cup \{x_{I(Q)}\}$ for all $P,Q\in \mathcal{P}$. In particular, if $x_{I(P)}=x_{I(Q)}$ for some $P,Q\in \mathcal{P}$, then $I(P)=I(Q)$, so $P=c(I(P))=c(I(Q))=Q$. Hence $\{x_{I(P)} : P \in \mathcal{P}\}$ is an independent set with size $s-1$, and $\bigcap_{P\in\mathcal{P}} I(P)=\{v,x\}$ for some $x\in \Sigma_2(v)$. Now $V\Delta=\bigcup_{P\in\mathcal{P}} P\subseteq\Sigma(v)\cap\Sigma(x)$, so $V\Delta=\Gamma(u)\cap\Gamma(v)\cap\Gamma(x)$. Suppose that $V\Delta=\Gamma(u)\cap\Gamma(v)\cap\Gamma(x')$ for some $x'\in \Sigma_2(v)$ where $x\neq x'$.
If $x$ is not adjacent to $x'$, then $\{v,x,x'\}$ is an independent set, so $x'=x_{I(P)}$ for some $P\in \mathcal{P}$, but then $\Sigma_2(x')=P\cup I(P)$, contradicting $P\subseteq V\Delta\subseteq \Sigma(x')$. Otherwise, $x$ is adjacent to $x'$, so $x$ and $x'$ have at least $|V\Delta|+(t-1)s$ common neighbours in $\Sigma$,  so $s(st-1)\leq \lambda(\Sigma)=s^2t-st-s+t$, a contradiction since $s\geq 2$.
\end{proof}

 We have seen that the complement $\Sigma$ of the point graph of a thick generalised quadrangle of order $(s,t)$ satisfies the following property: 
\begin{itemize}
\item[($\dag$)] $\Sigma$ is a strongly regular graph with $\mu(\Sigma)=(s-1)st$ for some $s\geq 2$ and $t\geq 2$ such that $\Sigma_2(v)$ induces $K_{(t+1)[s]}$ for all $v\in V\Sigma$.
\end{itemize}
Now suppose that $\Sigma$ is any graph that satisfies ($\dag$). Then $\overline{\Sigma}$ is locally $(t+1)\cdot K_s$, so $\lambda(\overline{\Sigma})=s-1$ and $b_1(\overline{\Sigma})=st$, and since $\mu(\Sigma)=(s-1)st$, we conclude that    $k_2(\overline{\Sigma})=s^2t$. Since $\overline{\Sigma}$ is strongly regular, $\mu(\overline{\Sigma}) k_2(\overline{\Sigma})=b_1(\overline{\Sigma})k_1(\overline{\Sigma})$, so   $\overline{\Sigma}$ has parameters $((st+1)(s+1),s(t+1),s-1,t+1)$, in which case   $\overline{\Sigma}$ is the point graph of a  generalised quadrangle of order $(s,t)$ by~\cite[Lemma 1.15.1]{BroCohNeu1989}. Thus we could replace assumption (ii) of Lemma~\ref{lemma:localGQ} with ($\dag$), which  is more in the spirit of this section, but we choose instead to give the more direct statement.

\section{Locally connected graphs}
\label{s:LC}

In this section, we prove Theorem \ref{thm:connected} as follows. By Lemma~\ref{lemma:localhom}, a locally connected $4$-CH graph is locally a connected $3$-homogeneous graph. For each of the $3$-homogeneous graphs $\Sigma$ described in Theorem \ref{thm:3hom}(ii)-(v), employing  a case-by-case analysis, we determine whether there are any $4$-CH graphs  that are locally $\Sigma$. In most cases, we show that there are no such graphs  using either Lemma \ref{lemma:xplus} and the results of \S\ref{s:uniquex}, or Lemmas \ref{lemma:x} and \ref{lemma:diam2} together with other methods. Then, in \S\ref{s:proof}, we combine all of these results to obtain the desired proof. 

\subsection{Locally the grid graph or its complement}
\label{s:grid}

Recall that the grid graph $K_n\Box K_n$ is a strongly regular graph with vertex set $VK_n\times VK_n$, where distinct  vertices $(u_1,u_2)$ and $(v_1,v_2)$ are adjacent whenever $u_1=v_1$ or $u_2=v_2$. Its complement is $K_n\times K_n$.

\begin{lemma}
\label{lemma:KnboxKn}
Let $\Gamma$ be a  locally $K_n\square K_n$ graph where $n\geq 3$. Then $\Gamma$ is not $4$-CH.
\end{lemma}

\begin{proof}
Every $\mu$-graph of $K_n\square K_n$ is isomorphic to $2\cdot K_1$, so $\Gamma$ has the unique $x$ property by Lemma~\ref{lemma:2K1}. 
For $u\in V\Gamma$ and $v\in \Gamma(u)$, the graph $\Gamma(u)\cap\Gamma_2(v)\simeq K_{n-1}\square K_{n-1}$, and this graph is neither edgeless nor complete, so $\Gamma$ is not $4$-CH  by Lemma~\ref{lemma:xplus}.
\end{proof}

\begin{lemma}
\label{lemma:KnxKnfirst}
Let $\Gamma$ be a  locally $K_n\times K_n$ graph where $n\geq 3$. Let $u,v,w\in V\Gamma$ be such that $v\in\Gamma(u)$ and $w\in \Gamma_2(u)\cap \Gamma(v)$. Then there exists  $x\in \Gamma(u)\cap\Gamma_2(v)$ such that $\Aut(\Gamma)_{u,v,w}\leq \Aut(\Gamma)_{x}$. 
\end{lemma}

\begin{proof}
 By assumption, $\Gamma(v)\simeq K_n\times K_n$, so  $\Gamma(u)\cap\Gamma(v)\simeq K_{n-1}\times K_{n-1}$. Let $\Delta$ be the graph induced by $\Gamma(u)\cap\Gamma(v)\cap\Gamma(w)$. Since  $u$ and $w$ are non-adjacent vertices in $\Gamma(v)$, it follows that  $\Delta\simeq K_{n-1}\times K_{n-2}$.
 In particular, there is a unique partition $\mathcal{P}$ of the vertices of $\Delta$  into $n-2$  independent sets of size $n-1$. 
Moreover, there is a unique way to extend $\mathcal{P}$ to a partition of $\Gamma(u)\cap\Gamma(v)$ into $n-1$ independent sets of size $n-1$. Thus $X:=\Gamma(u)\cap\Gamma(v)\cap\Gamma_2(w)$ is an independent set of size $n-1$.
Since $\Gamma(u)\simeq K_n\times K_n$ and $v\in \Gamma(u)$, there exists a unique vertex $x\in \Gamma(u)\cap\Gamma_2(v)$ such that $\{x\}\cup X$ is an independent set (of size $n$). If $g\in \Aut(\Gamma)_{u,v,w}$, then
 $x^g\in \Gamma(u)\cap\Gamma_2(v)$ and $\{x^g\}\cup X=(\{x\}\cup X)^g$ is  an independent set, so $x^g=x$, as desired.
\end{proof}

\begin{lemma}
\label{lemma:KnxKn}
Let $\Gamma$ be a  locally $K_n\times K_n$ graph where $n\geq 3$. Then $\Gamma$ is not $4$-CH.
\end{lemma}

\begin{proof}
For $u\in V\Gamma$ and $v\in \Gamma(u)$, the graph $\Gamma(u)\cap\Gamma_2(v)\simeq K_{n-1,n-1}$, and this graph is neither edgeless nor complete, so $\Gamma$ is not $4$-CH  by Lemmas~\ref{lemma:KnxKnfirst} and~\ref{lemma:x}.
\end{proof}

Note that for $n\geq 3$, the graph  $K_{n+1}\times K_{n+1}$  is locally $K_n\times K_n$ and $3$-homogeneous by Theorem~\ref{thm:3hom}, but it is not $4$-CH  by Lemma~\ref{lemma:KnxKn}.

\subsection{Locally the affine polar graph or its complement}
\label{s:affine}

Recall that the affine polar graph $VO^\varepsilon_{2m}(2)$ has vertex set $V_{2m}(2)$,  and  vectors $u$ and $v$ are adjacent whenever $Q(u+v)=0$, where $(V_{2m}(2),Q)$ is  a quadratic space with type  $\varepsilon$ (see \S\ref{s:defn} for the definition of a quadratic space).
The graph $VO^+_{2m}(2)$ is strongly regular with parameters 
$$(2^{2m}, (2^m-1)(2^{m-1}+1),2(2^{m-1}-1)(2^{m-2}+1),2^{m-1}(2^{m-1}+1)),$$
and the  graph $VO^-_{2m}(2)$ is strongly regular with parameters
$$(2^{2m}, (2^m+1)(2^{m-1}-1),2(2^{m-1}+1)(2^{m-2}-1),2^{m-1}(2^{m-1}-1))$$
(see~\cite[\S8 and Appendix~$C.12^\pm$]{Hub1975}).
In order to simplify the proofs in this section, we define a standard basis for a quadratic space. Let $f$ be the bilinear form associated with $Q$, and note that $f$ is alternating. By \cite[Proposition 2.5.3]{KleLie1990}, if 
 $\varepsilon=+$, then $V_{2m}(2)$ has a basis  $\{e_1,\ldots,e_m,f_1,\ldots,f_m\}$ where  $Q(e_i)=Q(f_i)=f(e_i,e_j)=f(f_i,f_j)=0$ and $f(e_i,f_j)=\delta_{i,j}$ for all $i,j$. Further, if $\varepsilon=-$, then 
 $V_{2m}(2)$ has a basis $\{e_1,\ldots,e_{m-1},f_1,\ldots,f_{m-1},x,y\}$ where  $Q(e_i)=Q(f_i)=f(e_i,e_j)=f(f_i,f_j)=f(e_i,x)=f(f_i,x)=f(e_i,y)=f(f_i,y)=0$ and $f(e_i,f_j)=\delta_{i,j}$ for all $i,j$, and $Q(x)=1$, $f(x,y)=1$ and $Q(y)=1$;  we then define $e_m:=x$ and $f_m:=y$. In either case, we refer to $\{e_1,\ldots,e_m,f_1,\ldots,f_m\}$ as a \textit{standard basis} of $V$.

\begin{lemma}
 \label{mu polar}
 For $m\geq 3$ and $\varepsilon\in\{+,-\}$, the graph  $\Sigma:=VO^\varepsilon_{2m}(2)$ is strongly regular with $\lambda(\Sigma)=\mu(\Sigma)-2$, and every $\mu$-graph of $\Sigma$ is isomorphic to a  graph $\Delta$ with diameter $3$ such that $k_1(\Delta)=k_2(\Delta)$ and $k_3(\Delta)=1$.
\end{lemma}

\begin{proof}
Let $(V,Q)$ be the  quadratic space associated with $\Sigma$, and let $\{e_1,\ldots,e_m,f_1,\ldots,f_m\}$ be a standard basis of $V$. As noted above, the  graph $\Sigma$ is strongly regular with  $\lambda(\Sigma)=\mu(\Sigma)-2$, and
since $\Sigma$ is $3$-CH by Theorem~\ref{thm:3hom}, every $\mu$-graph of $\Sigma$ is isomorphic to the  graph $\Delta$ induced by $\Sigma(0)\cap \Sigma(e_1+f_1)$, and  $\Delta$ is vertex-transitive. Thus $k_i(\Delta)$ is defined for all $i$. Let $W$ be the span of $\{e_2,\ldots,e_m,f_2,\ldots,f_m\}$, and let $I:=\{w\in W : Q(w)=0\}$. Note that $I$ contains non-zero vectors since $m\geq 3$.  Define
\begin{align*}
 X&:=\{e_1+w : w \in I\},\\
Y&:=\{f_1+w : w \in I\}.
\end{align*}
Now $V\Delta=X\cup Y$. Further, $\Delta(e_1)=X\setminus \{e_1\}$ and $\Delta(f_1)=Y\setminus \{f_1\}$. If $w\in I\setminus \{0\}$, 
then there exists $w'\in I\setminus \{0\}$ such that $f(w,w')=1$ by Lemma~\ref{lemma:quadratic}, in which case $f_1+w$ and $e_1+w'$ are adjacent. It follows that $\Delta_2(e_1)=Y\setminus \{f_1\}$ and $\Delta_3(e_1)=\{f_1\}$. Thus $\Delta$ has diameter $3$, $k_1(\Delta)=k_2(\Delta)$ and $k_3(\Delta)=1$.
\end{proof}

\begin{lemma}
\label{form}
 Let $\Sigma$ be  $VO^\varepsilon_{2m}(2)$ where $m\geq 3$ and $\varepsilon\in\{+,-\}$. If  $y$ and $z$ are distinct non-adjacent neighbours of $v\in V\Sigma$, then $\Sigma(v)\cap\Sigma(y)\cap\Sigma(z)\cap \Sigma(v+y+z)=\varnothing$.
\end{lemma}

\begin{proof}
Let $(V,Q)$ be the quadratic space associated with $\Sigma$, and let $f$ be the bilinear form that is  associated with $Q$.  If $u\in \Sigma(v)\cap\Sigma(y)\cap\Sigma(z)$, then 
 \begin{align*}
 Q(u+v+y+z)&=1+f(u+v,y+z)\\
 &=1+f(u,y)+f(u,z)+f(v,y)+f(v,z)\\
 &=1+Q(u)+Q(y)+Q(u)+Q(z)+Q(v)+Q(y)+Q(v)+Q(z)\\
 &=1.
\end{align*}
Thus  $\Sigma(v)\cap\Sigma(y)\cap\Sigma(z)\cap \Sigma(v+y+z)=\varnothing$.
\end{proof}

\begin{lemma}
\label{dist 3 polar}
 Let $\Sigma$ be  $VO^\varepsilon_{2m}(2)$ where $m\geq 3$ and $\varepsilon\in\{+,-\}$.  For any distinct non-adjacent $y,z\in V\Sigma$, if $v$ and $x$ are vertices at distance $3$ in the graph induced by $\Sigma(y)\cap\Sigma(z)$, then $y$ and $z$ are at distance $3$ in the graph induced by $\Sigma(v)\cap\Sigma(x)$.
\end{lemma}

\begin{proof}
Let $\Delta$ be the graph induced by $\Sigma(y)\cap\Sigma(z)$.
 Suppose that $v$ and $x$ are vertices at distance $3$ in $\Delta$.  Now $Q(v+y+z+y)=Q(v+z)=0$ since $v\in\Sigma(z)$, so $v+y+z\in \Sigma(y)$. Similarly, $v+y+z\in \Sigma(z)$ and $v+y+z\notin\Sigma(v)$. Thus  $v+y+z\in V\Delta\setminus \Delta(v)$.  By Lemma~\ref{mu polar}, $\diam(\Delta)=3$ and $k_3(\Delta)=1$, so  $\Delta_3(v)=\{x\}$. Since $v+y+z$  does not lie in $\Delta_2(v)$ by Lemma~\ref{form}, it follows that $x=v+y+z$. Let $\Lambda$ be the graph induced by $\Sigma(v)\cap\Sigma(x)$. Now $y,z\in V\Lambda$ and $v+x+z=y$. By Lemma~\ref{form}, $z$ and $y$ have no common neighbours in $V\Lambda$, so $z\in \Lambda_3(y)$.
\end{proof}

\begin{lemma}
\label{lemma:VO}
 Let $\Gamma$ be a locally $VO^\varepsilon_{2m}(2)$ graph where $m\geq 3$ and $\varepsilon\in\{+,-\}$. Then $\Gamma$ is not $4$-CH.
\end{lemma}

\begin{proof}
By Lemmas~\ref{mu polar} and~\ref{dist 3 polar}, the conditions of Lemma~\ref{lemma:polar} are satisfied, so $\Gamma$ has the unique $x$ property. 
For $u\in V\Gamma$ and $v\in \Gamma(u)$,  the graph induced by $\Gamma(u)\cap\Gamma_2(v)$ is neither edgeless nor complete, so $\Gamma$ is not $4$-CH by Lemma~\ref{lemma:xplus}.
\end{proof}

\begin{lemma}
 \label{mu polar bar}
Let $\Sigma$ be  $\overline{VO^\varepsilon_{2m}(2)}$
where $m\geq 3$ and $\varepsilon\in\{+,-\}$. Then $\Sigma$ is a strongly regular graph with $\mu(\Sigma)>0$, and every $\mu$-graph of $\Sigma$ is isomorphic to a regular graph $\Delta$, where for each $y\in V\Delta$, there exists a unique $z\in \Delta_2(y)$ such that $\Delta(y)=\Delta(z)$. Further, $\Sigma$ has valency $\mu(\Sigma)+2\lambda(\Sigma)-2k_1(\Delta)$.
\end{lemma}

\begin{proof}
Let $(V,Q)$ be the  quadratic space associated with $\Sigma$, and let $\{e_1,\ldots,e_m,f_1,\ldots,f_m\}$ be a standard basis of $V$. Let $f$ be the bilinear form that is associated with $Q$. The graph $\Sigma$ is strongly regular with  $\mu(\Sigma)>0$, and
since $\Sigma$ is $3$-CH by Theorem~\ref{thm:3hom}, every $\mu$-graph of $\Sigma$ is isomorphic to the  graph $\Delta$ induced by $\Sigma(0)\cap \Sigma(e_1)$, and  $\Delta$ is vertex-transitive and therefore regular.  Note that $V\Delta=\{x \in V : Q(x)=1,f(x,e_1)=0\}$.

First we claim that for each $y\in V\Delta$, there exists a unique $z\in \Delta_2(y)$ such that $\Delta(y)=\Delta(z)$. Since $\Delta$ is vertex-transitive, we may assume that $y=e_2+f_2$. Let $z:=e_1+e_2+f_2$. Now   $\Delta(y)=\Delta(z)$, so $z\in \Delta_2(y)$, and it remains to show that $z$ is the unique such vertex. Let $w\in V\Delta\setminus (\{y,z\}\cup\Delta(y))$. We wish to prove that $\Delta(y)\neq \Delta(w)$. Let $U:=\langle e_3,\ldots,e_m,f_3,\ldots,f_m\rangle$. Now $Q(w)=1$, $f(w,e_1)= 0$ and  $\delta_2:=f(w,e_2)=f(w,f_2)$, so $w=\delta_1 e_1+\delta_2 y+u$ for some $u\in U\setminus\{0\}$ where $\delta_1:=f(w,f_1)$.
 If $\delta_2 =0$, then $Q(u)=1$, so $e_2+u\in \Delta(y)\setminus \Delta(w)$. Otherwise, $\delta_2=1$ and  $Q(u)=0$.   By Lemma~\ref{lemma:quadratic}, there exists  $v\in U$ such that $Q(v)=1$ and $Q(u+v)=0$, so $e_2+v\in \Delta(y)\setminus \Delta(w)$, as desired.
 
 It remains to show that $|\Sigma(0)|=|V\Delta|+2\lambda(\Sigma)-2k_1(\Delta)$. Let $y:=e_2+f_2$ and $z:=e_1+e_2+f_2$. Now $\Sigma(0)\cap \Sigma(y)\cap \Sigma(z)=\Delta(y)=\Delta(z)$, so it  suffices to prove that every vertex in $\Sigma(0)\setminus \Sigma(e_1)$ is adjacent to exactly one of $y$ or $z$, for then $\Sigma(0)$ is the disjoint union of $V\Delta$ and the symmetric difference of $\Sigma(0)\cap \Sigma(y)$ and $\Sigma(0)\cap \Sigma(z)$. If $x\in  \Sigma(0)\setminus \Sigma(e_1)$, then $Q(x)=1$ and $f(x,e_1)=1$, so $Q(x+y)=1+Q(x+z)$. Thus $x$ is adjacent to exactly one of $y$ or $z$. \end{proof}

\begin{lemma}
\label{lemma:VOcomp}
 Let $\Gamma$ be a locally $\overline{VO^\varepsilon_{2m}(2)}$ graph where $m\geq 3$ and $\varepsilon\in\{+,-\}$. Then $\Gamma$ is not $4$-CH.
\end{lemma}

\begin{proof}
By Lemma~\ref{mu polar bar}, the conditions of Lemma~\ref{lemma:polarcomp} are satisfied, so $\Gamma$ has the unique $x$ property. 
For $u\in V\Gamma$ and $v\in \Gamma(u)$,  the graph induced by $\Gamma(u)\cap\Gamma_2(v)$ is neither edgeless nor complete, so $\Gamma$ is not $4$-CH by Lemma~\ref{lemma:xplus}.
\end{proof}

\subsection{Locally the point graph of $Q^-_5(q)$ or its complement}
\label{s:Q5}

The point graph of the generalised quadrangle $Q^-_5(q)$ for $q$ a prime power  is a strongly regular graph with parameters $(q^4+q^3+q+1,q^3+q,q-1,q^2+1)$ and automorphism group $\PGammaO^-_6(q)$. Its complement has parameters $(q^4+q^3+q+1,q^4,q(q^2+1)(q-1),q^3(q-1))$. 

\begin{lemma}
\label{lemma:Q5hasdist2V4}
Let $\Sigma$ be the point graph of $Q_5^-(q)$ for $q$ a prime power. Let $v\in V\Sigma$. Then $\Sigma_2(v)$ induces the affine polar graph $VO_4^-(q)$. In particular, $\Sigma_2(v)$ is connected.
\end{lemma}

\begin{proof}
Let $(V,Q)$ be the quadratic space associated with $\Sigma$, and let $f$ be the bilinear form that is associated with $Q$. Now $v=\langle e_1\rangle$ where  $e_1$ is a non-zero singular vector in $V$.  By~\cite[Proposition~2.5.3]{KleLie1990}, there exists a non-zero singular vector $f_1$ in $V$ such that $f(e_1,f_1)=1$ and $V=\langle e_1,f_1\rangle\oplus W$, where $W:=\langle e_1,f_1\rangle^\perp\simeq V_4(q)$ and $(W,Q|_W)$ is a quadratic space with the same type as $(V,Q)$, namely minus type. Observe that $\Sigma_2(v)=\{\langle -Q(w)e_1+f_1+w\rangle : w\in W\}$. Further, if $x_1:=-Q(w_1)e_1+f_1+w_1$ and $x_2:= -Q(w_2)e_1+f_1+w_2$ where $w_1,w_2\in W$, then $Q(x_1-x_2)=Q(w_1-w_2)$, so $\langle x_1\rangle $ and $\langle x_2\rangle$ are adjacent vertices of $\Sigma_2(v)$ if and only if $w_1$ and $w_2$ are adjacent vertices in the affine polar graph corresponding to the  quadratic space $(W,Q|_W)$. Thus $\Sigma_2(v)\simeq VO_4^-(q)$. Now $VO_4^-(q)$ is strongly regular with $\mu\neq 0$ (see~\cite[\S8 and Appendix~$C.12^\pm$]{Hub1975}), so $\Sigma_2(v)$ is connected.
\end{proof}

Recall that the McLaughlin graph is a strongly regular graph with parameters $(275,112,30,56)$ and automorphism group $\McL{:}2$; it is also locally the point graph of $Q^-_5(3)$. 

\begin{lemma}
\label{lemma:pointgraph}
Let $\Gamma$ be a connected graph that is locally the point graph of $Q^-_5(q)$ for $q$ a prime power. Then $\Gamma$ is $(G,4)$-CH if and only if  $q=3$, $\Gamma$ is the McLaughlin graph and $G=\Aut(\Gamma)$.
\end{lemma}

\begin{proof}
Suppose that $\Gamma$ is $(G,4)$-CH. Let $\Sigma$ be the point graph of $Q^-_5(q)$. For $v\in V\Sigma$ and $x\in \Sigma_2(v)$, the set $\Sigma(v)\cap\Sigma(x)$ induces the graph $(q^2+1)\cdot K_1$ since $Q^-_5(q)$ is a generalised quadrangle, and $\Sigma_2(v)$ induces the  connected graph $VO_4^-(q)$ by Lemma~\ref{lemma:Q5hasdist2V4}. Thus $\diam(\Gamma)=2$ by  Lemma~\ref{lemma:diam2}. It follows that $\Gamma$ is $(G,2)$-homogeneous, so 
 $G$ is primitive of rank $3$ on $V\Gamma$ by Lemma \ref{lemma:2hom}.

First suppose that $q=2$. Since $G$ is primitive of rank $3$ and $\Gamma$ is locally $Q^-_5(2)$, it follows from~\cite[Proposition 1 and Theorem 2]{BueHub1977} 
that $\Gamma$ is isomorphic to the affine polar graph $VO_6^-(2)$.  Let $(V_6(2),Q)$ be the quadratic space associated with $\Gamma$. Let $u$ denote the vertex $0$, and  note that $G_u\leq \Aut(\Gamma)_u=\SO_6^-(2)$. 
 Choose $v\in \Gamma(u)$ and $w\in \Gamma_2(u)\cap\Gamma(v)$.
 Now $Q(v)=0$, $Q(w)\neq 0$ and $Q(v+w)=0$. Thus $x:=v+w\in \Gamma(u)\cap\Gamma_2(v)$, and any element of $G_{u,v,w}$ must also fix $x$, but then the graph induced by $\Gamma(u)\cap\Gamma_2(v)$ is either edgeless or complete by Lemma~\ref{lemma:x}, a contradiction.

Thus $q\geq 3$. Let $u\in V\Gamma$ and $H:=G_u^{\Gamma(u)}$. Since $\Gamma$ is $(G,4)$-CH, the graph induced by $\Gamma(u)$ is $(H,3)$-homogeneous by Lemma~\ref{lemma:localhom}. In particular, $H$ is  transitive of rank $3$, so $\POmega_6^-(q)\unlhd H\leq \PGammaO_6^-(q)$ by~\cite[Theorem~1.3]{KanLie1982}.
Since $\Gamma$ is locally $Q^-_5(q)$, it follows from~\cite[Proposition 1 and Theorem 4]{BueHub1977} that $q=3$ and  $\Gamma$ is isomorphic to the McLaughlin graph with $\McL\leq G\leq \McL{:}2$. Let $v\in \Gamma(u)$ and $x\in \Gamma(u)\cap \Gamma_2(v)$.  Now $G_{u,v}$ is transitive of rank $3$ on $\Gamma(u)\cap \Gamma_2(v)$ since $\Gamma$ is $(G,4)$-CH, but if $G=\McL$, then $G_u\simeq \POmega_6^-(3)$ and $G_{u,v}\simeq q^4{:}A_6$, in which case $G_{u,v}$ is transitive of rank $4$, a contradiction. Thus $G=\McL{:}2=\Aut(\Gamma)$.

Conversely, it is routine to verify that the McLaughlin graph is $4$-CH using {\sc  Magma}~\cite{Magma} and~\cite{web-atlas}
 (see Remark~\ref{remark:magma}).
\end{proof}

\begin{lemma}
\label{lemma:pointgraphcomp}
Let $\Gamma$ be locally the complement of the point graph of $Q^-_5(q)$ for $q$  a prime power. Then $\Gamma$ is not $4$-CH.
\end{lemma}

\begin{proof}
By  Lemma~\ref{lemma:localGQ}, $\Gamma$ has the unique $x$ property. 
For $u\in V\Gamma$ and $v\in \Gamma(u)$,  the graph induced by $\Gamma(u)\cap\Gamma_2(v)$ is neither edgeless nor complete, so $\Gamma$ is not $4$-CH by Lemma~\ref{lemma:xplus}.
\end{proof}

\subsection{Locally the pentagon}
\label{s:C5}

The icosahedron is locally $C_5$, and it is straightforward to prove that it is the unique such  connected graph  (see~\cite[Lemma 9]{Gar1978}). Note that  the graph $C_5$ is self-complementary. 

\begin{lemma}
\label{lemma:C5}
Let $\Gamma$ be a connected locally $C_5$ graph. Then $\Gamma$ is isomorphic to the icosahedron, and $\Gamma$ is not $4$-CH.
\end{lemma}

\begin{proof}
The icosahedron $\Gamma$ contains a $3$-geodesic $(u,v,w,x)$, and there exists $y\in \Gamma_2(u)\cap\Gamma_2(v)\cap\Gamma(w)$,  so  if $\Gamma$ is $4$-CH, then there exists $g\in \Aut(\Gamma)_{u,v,w}$ such that $x^g=y$,  a contradiction.
\end{proof}

\subsection{Locally the Clebsch graph or the folded $5$-cube}
\label{s:Clebsch}

Recall that the Clebsch graph is the halved $5$-cube, and its complement is isomorphic to the folded $5$-cube $\Box_5$.
The Clebsch graph is a strongly regular graph with parameters $(16,10,6,6)$. Its complement has parameters $(16,5,0,2)$. 
\begin{lemma}
\label{lemma:folded5}
Let $\Gamma$ be a locally $\Box_5$ graph. Then $\Gamma$ is not $4$-CH.
\end{lemma}

\begin{proof}
Any $\mu$-graph of $\Box_5$ is isomorphic to $2\cdot K_1$, so $\Gamma$ has the unique $x$ property by  Lemma~\ref{lemma:2K1}. 
For $u\in V\Gamma$ and $v\in \Gamma(u)$, the graph induced by $\Gamma(u)\cap\Gamma_2(v)$ is neither edgeless nor complete, so $\Gamma$ is not $4$-CH  by Lemma~\ref{lemma:xplus}.
\end{proof}

The Schl\"{a}fli graph is a locally Clebsch graph that is $4$-homogeneous \cite{Buc1980} and therefore $4$-CH. The   Schl\"{a}fli graph is isomorphic to the complement of the  point graph of $Q^-_5(2)$, and its  automorphism group is $\PSO^-_6(2)$. It is a strongly regular graph with  parameters  $(27,16,10,8)$, and it is the unique such graph  (see~\cite[Lemma 10.9.4]{GodRoy2001}). 

\begin{lemma}
\label{lemma:clebsch}
 Let $\Gamma$ be a connected locally Clebsch graph. Then  $\Gamma$ is $(G,4)$-CH if and only if $\Gamma$ is the Schl\"{a}fli graph and $\POmega^-_6(2)\leq G\leq \PSO^-_6(2)$.
 \end{lemma}
 
\begin{proof}
 Suppose that $\Gamma$ is $(G,4)$-CH. Any $\mu$-graph of the Clebsch graph $\Sigma$ is isomorphic to $K_{3[2]}$, and    $\Sigma_2(v)$ induces $K_5$ for all $v\in V\Sigma$,  
  so $\diam(\Gamma)=2$ by Lemma~\ref{lemma:diam2}. Let $u\in V\Gamma$, $v\in\Gamma(u)$ and $w\in\Gamma_2(u)\cap\Gamma(v)$.  Since  $K_{3[2]}$ is regular with diameter 2 and $k_2=1$, Lemma~\ref{lemma:k2=1} implies that there exists a unique $x\in \Gamma(u)\cap\Gamma_2(v)$ such that $\Gamma(u)\cap\Gamma(v)\cap\Gamma(w)=\Gamma(u)\cap\Gamma(v)\cap\Gamma(x)$.
 Now $\Gamma(u)\cap\Gamma_2(v)\cap\Gamma(w)$ is either $\{x\}$ or $\Gamma(u)\cap\Gamma_2(v)\setminus \{x\}$ by Lemma~\ref{lemma:x}, so $c_2(\Gamma)$ is either $1+6+1=8$ or $1+6+4=11$. However, $c_2(\Gamma)k_2(\Gamma)=b_1(\Gamma)k_1(\Gamma)=5\cdot 16=80$, so  $c_2(\Gamma)=8$ and $k_2(\Gamma)=10$. Now $\Gamma$ is a strongly regular graph with parameters $(27,16,10,8)$, so $\Gamma$ is isomorphic to the Schl\"{a}fli graph, and $\POmega^-_6(2)\leq G\leq \PSO^-_6(2)$ since $G$ has rank $3$ on $V\Gamma$. 

Conversely, if $\Gamma$ is the Schl\"{a}fli graph and $\POmega^-_6(2)\leq G\leq \PSO^-_6(2)$, then it is routine to verify that $\Gamma$ is $(G,4)$-CH.
\end{proof}

\subsection{Locally the Higman-Sims graph or its complement}
\label{s:HS}

Recall that the Higman-Sims graph is a strongly regular graph with parameters $(100,22,0,6)$. Its automorphism group is  $\HS{:}2$, and the stabiliser of a vertex is $\M_{22}{:}2$.  The  complement of the Higman-Sims graph has parameters $(100,77,60,56)$.

\begin{lemma}
\label{lemma:HS}
Let $\Gamma$ be a connected locally $\Sigma$  graph where $\Sigma$ is the Higman-Sims graph or its complement. Then $\Gamma$ is not $4$-CH.
 \end{lemma}

\begin{proof}
Let $G:=\Aut(\Gamma)$ and $u\in V\Gamma$, and let $H:=G_u^{\Gamma(u)}$. Suppose for a contradiction that $\Gamma$ is $4$-CH. The graph $\Sigma$ induced by $\Gamma(u)$ is $(H,3)$-homogeneous by Lemma~\ref{lemma:localhom}. In particular, $H$ is transitive of rank~$3$, so $\HS\leq H\leq \HS{:}2$. Let $v\in \Gamma(u)$, and note that $\Sigma(v)=\Gamma(u)\cap \Gamma(v)$ and $\Sigma_2(v)=\Gamma(u)\cap\Gamma_2(v)$.

First we claim that $G_u$ acts faithfully on $V\Sigma$. Suppose that $g\in G_{u,v}$ fixes $\Sigma(v)$ pointwise. By Lemma~\ref{lemma:faithful}, it suffices to show that $g$ fixes $\Sigma_2(v)$ pointwise. Let $x\in \Sigma_2(v)$,  and suppose for a contradiction that $x\neq x^g$. Note that $x^g\in \Sigma_2(v)$. Now $x$ and $x^g$ have  $\mu(\Sigma)$ common neighbours in $\Sigma(v)$ since $g$ fixes $\Sigma(v)$ pointwise. If $\Sigma$ is the Higman-Sims graph, then $x$ and $x^g$ are not adjacent since $\lambda(\Sigma)=0$, but $\Sigma_2(v)$ induces a strongly regular graph with parameters $(77,16,0,4)$ by~\cite{HigSim1968}, so $x$ and $x^g$ have $6+4$ common neighbours, contradicting $\mu(\Sigma)=6$. Hence $\Sigma$ is the complement of the Higman-Sims graph. Now  $\Sigma_2(v)$ induces $K_{22}$, so  $x$ and $x^g$ have $56+20$ common neighbours,  contradicting $\lambda(\Sigma)=60$. 

Thus $G_u$ is isomorphic to $\HS$ or $\HS{:}2$, and $G_{u,v}$ is isomorphic to $\M_{22}$ or $\M_{22}{:}2$ respectively. Note that $b_1(\Gamma)=|\Sigma_2(v)|$, so $b_1(\Gamma)= 77$ when $\Sigma$ is the Higman-Sims graph and $22$ otherwise.  Let $w\in \Gamma_2(u)\cap\Gamma(v)$ and $y\in \Gamma(u)\cap\Gamma_2(v)$. Since $G_{u,v}$ acts transitively on $\Gamma_2(u)\cap\Gamma(v)$ and $\Gamma(u)\cap\Gamma_2(v)$, the stabilisers $G_{u,v,w}$ and $G_{u,v,y}$ have index $b_1(\Gamma)$ in $G_{u,v}$. By~\cite{Atlas}, 
 $G_{u,v}$ has   unique classes of maximal subgroups of index $22$ and $77$, and if $M$ is a maximal subgroup of $G_{u,v}$ of index at most $77$ not lying in either of these classes, then $G_{u,v}=\M_{22}{:}2$ and $M=\M_{22}$. We conclude that $G_{u,v,w}$ and $G_{u,v,y}$ are maximal subgroups of $G_{u,v}$ and therefore  conjugate in $G_{u,v}$. 

Thus there exists $x\in \Gamma(u)\cap\Gamma_2(v)$ such that $G_{u,v,w}=G_{u,v,x}$. By Lemma~\ref{lemma:x}, the graph induced by $\Gamma(u)\cap\Gamma_2(v)$ is either edgeless or complete, so $\Sigma$ is the complement of the Higman-Sims graph. Now $|\Gamma(u)\cap\Gamma_2(v)\cap\Gamma(w)|=1$ or 21 by Lemma~\ref{lemma:x}, so $c_2(\Gamma)=1+56+1$ or $1+56+21$.  However, $c_2(\Gamma)k_2(\Gamma)=b_1(\Gamma)k_1(\Gamma)=22\cdot 100$, a contradiction.
\end{proof}

\subsection{Locally the McLaughlin graph or its complement}
\label{s:McL}

Recall that the McLaughlin graph is a strongly regular graph with parameters $(275,112,30,56)$. Its  automorphism group is $\McL{:}2$, and the stabiliser of a vertex is $\PSU_4(3){:}2$. The complement of the  McLaughlin graph has parameters $(275,162,105,81)$. 

\begin{lemma}
\label{lemma:McL}
 Let $\Gamma$ be a  connected locally $\Sigma$  graph where $\Sigma$ is the McLaughlin graph or its complement. Then $\Gamma$ is not $4$-CH.
\end{lemma}

\begin{proof}
Let $G:=\Aut(\Gamma)$ and $u\in V\Gamma$, and let $H:=G_u^{\Gamma(u)}$. Suppose for a contradiction that $\Gamma$ is $4$-CH. The graph $\Sigma$ induced by $\Gamma(u)$ is $(H,3)$-homogeneous by Lemma~\ref{lemma:localhom}. In particular, $H$ is transitive of rank~$3$, so $\McL\leq H\leq \McL{:}2$. Let $v\in \Gamma(u)$, and note that $\Sigma(v)=\Gamma(u)\cap \Gamma(v)$ and $\Sigma_2(v)=\Gamma(u)\cap\Gamma_2(v)$.

First we claim that $G_u$ acts faithfully on $V\Sigma$. Suppose that $g\in G_{u,v}$ fixes $\Sigma(v)$ pointwise. By Lemma~\ref{lemma:faithful}, it suffices to show that $g$ fixes $\Sigma_2(v)$ pointwise. Let $x\in \Sigma_2(v)$,  and suppose for a contradiction that $x\neq x^g$. Note that $x^g\in \Sigma_2(v)$. Now $x$ and $x^g$ have $\mu(\Sigma)$ common neighbours in $\Sigma(v)$. If $\Sigma$ is the McLaughlin graph, then $x$ and $x^g$ are not adjacent since $\lambda(\Sigma)<\mu(\Sigma)$, but $\Sigma_2(v)$ induces a strongly regular graph with parameters $(162,56,10,24)$, so $x$ and $x^g$ have $56+24$ common neighbours, contradicting $\mu(\Sigma)=56$. Hence $\Sigma$ is the complement of the McLaughlin graph. Now  $\Sigma_2(v)$ induces a strongly regular graph with parameters $(112,81,60,54)$, so either $x$ and $x^g$ are adjacent with $81+60$ common neighbours, or $x$ and $x^g$ are not adjacent with $81+54$ common neighbours, both of which are contradictions since $\lambda(\Sigma)=105$ and $\mu(\Sigma)=81$. 

Thus $G_u$ is isomorphic to $\McL$ or $\McL{:}2$, and $G_{u,v}$ is isomorphic to $\PSU_4(3)$ or $\PSU_4(3){:}2_3$ (in the notation of~\cite{Atlas}) respectively. Note that $b_1(\Gamma)=|\Sigma_2(v)|$, so $b_1(\Gamma)= 162$ when $\Sigma$ is the McLaughlin graph and $112$ otherwise. Let $w\in \Gamma_2(u)\cap\Gamma(v)$ and $y\in \Gamma(u)\cap\Gamma_2(v)$. Since $G_{u,v}$ acts transitively on $\Gamma_2(u)\cap\Gamma(v)$ and $\Gamma(u)\cap\Gamma_2(v)$, the stabilisers  $G_{u,v,w}$ and $G_{u,v,y}$ have index $b_1(\Gamma)$ in $G_{u,v}$.

Suppose that $\Sigma$ is the complement of the McLaughlin graph. By~\cite{Atlas}, $G_{u,v}$ has a    unique class of maximal subgroups of index $112$, and if $M$ is a maximal subgroup of $G_{u,v}$ of index at most $112$ not lying in this class, then $G_{u,v}=\PSU_4(3){:}2_3$ and $M=\PSU_4(3)$. We  conclude that $G_{u,v,w}$ and $G_{u,v,y}$ are maximal subgroups of $G_{u,v}$ and therefore  conjugate in $G_{u,v}$. Thus there exists $x\in \Gamma(u)\cap\Gamma_2(v)$ such that $G_{u,v,w}=G_{u,v,x}$, but then the graph induced by $\Gamma(u)\cap\Gamma_2(v)$ is either  edgeless  or complete by Lemma~\ref{lemma:x}, a contradiction.

Hence $\Sigma$ is the McLaughlin graph. By~\cite{Atlas},   $G_{u,v,w}$ is a maximal subgroup of $G_{u,v}$ of index 162, and either $G_{u,v}=\PSU_4(3)$ and $G_{u,v,w}\simeq \PSL_3(4)$ (there are two classes), or $G_{u,v}=\PSU_4(3){:}2_3$ and $G_{u,v,w}$ is isomorphic to   $\PSL_3(4){:}2_1$ or $\PSL_3(4){:}2_3$. 
Let  $X_1:=\Gamma(u)\cap\Gamma_2(v)\cap\Gamma(w)$ and $X_2:=(\Gamma(u)\cap\Gamma_2(v))\setminus\Gamma(w) $. Fix $i\in \{1,2\}$.  We claim that either $|X_i|=2$, or $|X_i|$ is divisible by one of $21$, $56$ or $120$. Clearly the claim holds if $|X_i|=0$, so we may choose $x\in X_i$. Now $G_{u,v,w,x}$ is a proper subgroup of $G_{u,v,w}$ by Lemma~\ref{lemma:x} since the graph induced by $\Gamma(u)\cap\Gamma_2(v)$ is neither complete nor edgeless, so $G_{u,v,w,x}$ is contained in a maximal subgroup $M$ of $G_{u,v,w}$. Further, $G_{u,v,w}$ acts transitively on $X_i$ since $\Gamma$ is $4$-CH, so  $[G_{u,v,w}:M][M:G_{u,v,w,x}]= |X_i|$. Recall that $|X_i|\leq 162$.  If $G_{u,v,w}=\PSL_3(4)$, then   $[G_{u,v,w}:M]$ is $21$, $56$ or $120$ by~\cite{Atlas}, so the claim holds. Otherwise, $G_{u,v,w}$ is $\PSL_3(4){:}2_1$ or $\PSL_3(4){:}2_3$. Then $[G_{u,v,w}:M]$ is $2$,  $56$, $105$ or $120$ by~\cite{Atlas}, so the claim holds unless $[G_{u,v,w}:M]=2$, in which case $M=\PSL_3(4)$. If $|X_i|=2$, then the claim holds; otherwise $G_{u,v,w,x}$ is a proper subgroup of $M$, so $[M:G_{u,v,w,x}]$ is divisible by $21$, $56$ or $120$, proving the claim. 

Since $|X_1|+|X_2|=162$, it follows from the claim that each of $|X_1|$ and $|X_2|$ is divisible by one of $21$, $56$ or $120$. Observe that if  $21a+56b+120c=162$ for some non-negative integers $a,b,c$, then $a=2$, $b=0$ and $c=1$. Thus $\{|X_1|,|X_2|\}=\{42,120\}$. Since $c_2(\Gamma)=1+56+|X_1|$ and $c_2(\Gamma)k_2(\Gamma)=b_1(\Gamma)k_1(\Gamma)=162\cdot 275$, it follows that $|X_1|=42$ and $|X_2|=120$. Hence $c_2(\Gamma)=99$ and $k_2(\Gamma)=450$.   The graph $\Gamma$ has diameter 2 by Lemma~\ref{lemma:diam2}, so it is  strongly regular with parameters $(726,275,112,99)$. But then the polynomial $X^2-13X-176$ has integer roots by~\cite[Theorem 1.3.1]{BroCohNeu1989}, a contradiction.
\end{proof}

\subsection{Proof of Theorem~\ref{thm:connected}}
\label{s:proof}
Let $\Gamma$ be a locally finite, connected, locally connected $4$-CH graph. In particular, $\Gamma$ has non-zero valency, so   by Lemma~\ref{lemma:localhom}, $\Gamma$ is locally  a finite connected $3$-homogeneous graph $\Sigma$, and $\Sigma$ is  described in Theorem \ref{thm:3hom}. If Theorem \ref{thm:3hom}(i) holds, then $\Sigma$ is either $K_s$ for some $s\geq 1$, or $K_{(t+1)[s]}$ for some $t\geq 1$ and $s\geq 2$. In the former case, $\Gamma\simeq K_{s+1}$, and in the latter, $\Gamma\simeq K_{(t+2)[s]}$ (see, for example,~\cite[Lemma~7]{Gar1978}), so Theorem~\ref{thm:connected}(i) holds. Otherwise, one of Theorem \ref{thm:3hom}(ii)--(v) holds.
 By Lemmas~\ref{lemma:KnboxKn},  \ref{lemma:KnxKn}, \ref{lemma:VO}, \ref{lemma:VOcomp},  \ref{lemma:pointgraph}, \ref{lemma:pointgraphcomp}, 
\ref{lemma:C5}, \ref{lemma:folded5}, \ref{lemma:clebsch}, \ref{lemma:HS} and \ref{lemma:McL}, $\Gamma$ is the Schl\"{a}fli graph or the McLaughlin graph.  The Schl\"{a}fli graph is $4$-homogeneous~\cite{Buc1980}, but it is not $5$-CH, or else its local graph, the Clebsch graph, is $4$-homogeneous by Lemma~\ref{lemma:localhom}, a contradiction. Similarly, the McLaughlin graph is $4$-CH by Lemma~\ref{lemma:pointgraph}, but it is not $5$-CH, or else its local graph, the point graph of $Q^-_5(3)$, is $4$-homogeneous, a contradiction. Thus one of Theorem~\ref{thm:connected}(ii) or~(iii) holds.

\section{Locally disconnected graphs with girth $3$ and $c_2>1$}
\label{s:GQ}

In this section, we prove Theorem~\ref{thm:girth3quad}, which describes the locally finite locally disconnected $4$-CH graphs with girth $3$ and $c_2>1$. Recall that a locally finite $3$-CH graph $\Gamma$ is locally disconnected with girth $3$ if and only if $\Gamma$ is locally $(t+1)\cdot K_s$ for some integers $t\geq 1$ and $s\geq 2$ (see Lemma~\ref{lemma:3CH}).

The arguments in the next two lemmas are similar to those in the proof of~\cite[Lemma~6]{Gar1978}.

\begin{lemma}
\label{lemma:girth3c2>1finite}
 Let $\Gamma$ be a connected $4$-CH graph that is locally $(t+1)\cdot K_s$ where $t\geq 1$,  $s\geq 2$ and $c_2(\Gamma)>1$. Then $\diam(\Gamma)=2$ and $c_2(\Gamma)=t+1$. In particular,   $\Gamma$ is the point graph of a finite distance-transitive generalised quadrangle of order $(s,t)$.
\end{lemma}

\begin{proof}
Let $G:=\Aut(\Gamma)$.  Suppose for a contradiction  that $\diam(\Gamma)>2$.
Now there exists a geodesic $(u,v,w,x)$ in $\Gamma$, and there exists  $y\in \Gamma(u)\cap\Gamma(w)\setminus\{v\}$. Note that $y$ is not adjacent to $v$ by Lemma~\ref{lemma:notinduced}. Let $\mathcal{C}_v$ be the clique of size $s$ in $\Gamma_2(u)\cap\Gamma(v)$ that contains $w$, and let $\mathcal{C}_y$ be the clique of size $s$ in $\Gamma_2(u)\cap\Gamma(y)$ that contains $w$. If there exists $z\in \mathcal{C}_v\cap\mathcal{C}_y$ besides $w$, then $\{v,y,w,z\}$ induces a complete graph with one edge removed, contradicting Lemma~\ref{lemma:notinduced}. Thus $\mathcal{C}_v\cap\mathcal{C}_y=\{w\}$. Since $s\geq 2$, there exists $z\in \mathcal{C}_y\setminus \{w\}$, and note that $z\notin\Gamma(v)$, or else $z\in \mathcal{C}_v$. Since $z\in \Gamma_2(u)$, the sets $\{u,v,w,z\}$ and $\{u,v,w,x\}$ induce path graphs, so there exists $g\in G_{u,v,w}$ with $x^g=z$, but $x\in \Gamma_3(u)$,  a contradiction.

Thus $\diam(\Gamma)=2$.
Let $u\in V\Gamma$ and $w\in \Gamma_2(u)$. 
It follows from Lemma~\ref{lemma:notinduced} that $\Gamma(u)\cap\Gamma(w)$ induces $c_2\cdot K_1$, so $c_2\leq t+1$. Suppose for a contradiction that $c_2<t+1$.  Let $v\in \Gamma(u)\cap\Gamma(w)$ and let $\Delta$ be the graph induced by $\Gamma_2(u)\cap\Gamma_2(v)\cap\Gamma(w)$. Now $\Delta$ is  a disjoint union of the graphs $(c_2-1)\cdot K_{s-1}$ and $(t+1-c_2)\cdot K_s$, so $\Delta$ is not a vertex-transitive graph, but $G_{u,v,w}$ acts transitively on $V\Delta$,  a contradiction. Thus $c_2=t+1$.

Now $\Gamma$ is a strongly regular graph with parameters $((st+1)(s+1),(t+1)s,s-1,t+1)$.  Since $\Gamma$ is  locally $(t+1)\cdot K_s$, it is the point graph of a generalised quadrangle  of order $(s,t)$ by \cite[Lemma 1.15.1]{BroCohNeu1989}, and this generalised quadrangle   is distance-transitive since $\Gamma$ is $3$-CH.
\end{proof}

 Before we prove Theorem~\ref{thm:girth3quad}, we establish some further restrictions on the structure of $k$-CH graphs that are locally $c_2\cdot K_s$ for $k\geq 3$.

\begin{lemma}
\label{lemma:sdividest}
Let $\Gamma$ be a $k$-CH graph that is locally $(t+1)\cdot K_s$ where $k\geq 3$, $t\geq 1$, $s\geq 1$ and $c_2(\Gamma)=t+1$. Let $m:=\min(t+1,k-1)$ and $u\in V\Gamma$. Then any $m$ pairwise non-adjacent neighbours of  $u$ have a common neighbour in $\Gamma_2(u)$, and $s^{m-2}$ divides $t$.
\end{lemma}

\begin{proof}
 Let $M$ be a subset of $\Gamma(u)$ that induces the graph  $m\cdot K_s$. Let $X$ be the set of subsets of $M$ that induce the graph $m\cdot K_1$. 
   By Lemma~\ref{lemma:notinduced} and our assumption that $c_2(\Gamma)=t+1$, if $w\in \Gamma_2(u)$, then $\Gamma(w)\cap M\in X$. Thus there is a map $\varphi:\Gamma_2(u)\to X$ defined by  $w\mapsto \Gamma(w)\cap M$ for all $w\in \Gamma_2(u)$. Choose $w\in \Gamma_2(u)$, and let $X_1:=\Gamma(w)\cap M$.  If $X_2\in X$, then since $m\leq k-1$, there exists $g\in \Aut(\Gamma)_u$ such that $X_1^g=X_2$, so $X_2=\varphi(w^g)$, and  $g$ maps $\varphi^{-1}(X_1)$ to $\varphi^{-1}(X_2)$. Thus $\varphi$ is a surjective map whose preimages all have the same size; in particular, $s^{m}$ divides $|\Gamma_2(u)|= ts(t+1)s/(t+1)=ts^2$. We have also shown that any $m$ pairwise non-adjacent neighbours of  $u$ have a common neighbour in $\Gamma_2(u)$. 
\end{proof}

\begin{proof}[Proof of Theorem $\ref{thm:girth3quad}$]
Recall that $\Gamma$ is a locally finite, connected, locally disconnected graph with girth $3$ for which $c_2>1$, and suppose that $\Gamma$ is $4$-CH. Now $\Gamma$ is locally $(t+1)\cdot K_s$ for some $t\geq 1$ and $s\geq 2$. Let $m:=\min(t+1,3)$. By Lemmas~\ref{lemma:girth3c2>1finite} and \ref{lemma:sdividest}, $\Gamma$ is the point graph of a distance-transitive generalised quadrangle $\mathcal{Q}$ of order $(s,t)$, where  $s^{m-2}$ divides $t$. If $t=1$, then  by~\cite[\S 6.5]{BroCohNeu1989}, $\Gamma\simeq L(K_{s+1,s+1})\simeq K_{s+1}\Box K_{s+1}$, so   Theorem~\ref{thm:girth3quad}(i) holds. Thus we may assume that $t\geq 2$. Then $m=3$, so $s$ divides $t$. By Theorem~\ref{thm:GQ}, $s$ is a power of a prime, and 
we may assume that $\mathcal{Q}$ is one of the following: $W_3(s)$, $H_4(s)$,  $Q_4(s)$ for $s$ odd, or $Q^-_5(s)$. Recall that the points of $\mathcal{Q}$ are the one-dimensional totally singular subspaces of a symplectic, unitary or quadratic space  on $V:=V_d(s)$, where $d$ is $4$, $5$, $5$ or $6$ respectively, and the lines of $\mathcal{Q}$ are the two-dimensional totally singular subspaces of $V$. Further, $\Aut(\Gamma)\leq \PGammaL_d(s)$. 

We claim that $s\in \{2,3,4\}$. Choose a line $\ell$ of $\mathcal{Q}$, and let $u_1$ and $u_2$ be distinct points on $\ell$.  For $i\in \{1,2\}$, there exists $v_i\in V$ such that $u_i=\langle v_i\rangle $. Let $u_3:=\langle v_1+v_2\rangle$, and note that $u_3\in \ell$.  Since $\Gamma$ is $4$-CH and any four points on $\ell$ induce $K_4$, it follows that $\GammaL_d(s)_{u_1,u_2,u_3}$ acts transitively on the points of $\ell\setminus \{u_1,u_2,u_3\}$. For $g\in \GammaL_d(s)_{u_1,u_2,u_3}$, there exists $\sigma\in\Aut(\mathbb{F}_s)$ such that $(\lambda x)^g=\lambda^\sigma x^g$ for all $x\in V$ and $\lambda\in\mathbb{F}_s$, and $v_1^g=\mu v_1$ for some $\mu \in \mathbb{F}_s$, so  $v_2^g=\mu v_2$, and  $\langle v_1+\lambda v_2\rangle^g=\langle v_1+\lambda^\sigma v_2\rangle$ for all $\lambda\in \mathbb{F}_s$. It follows that $\Aut(\mathbb{F}_s)$ acts transitively on $\mathbb{F}_s\setminus \{0,1\}$. Thus $s\in \{2,3,4\}$, proving the claim.

Suppose that either $\mathcal{Q}=W_3(s)$ where $s\in \{2,3,4\}$, or $\mathcal{Q}=H_4(4)$. Let $(V,f)$ be the corresponding symplectic or unitary space. By \cite[Propositions 2.3.2 and 2.4.1]{KleLie1990}, $V$ has a basis
 $\{e_1,e_2,f_1,f_2\}$ or $\{e_1,e_2,f_1,f_2,x\}$ where  $f(e_i,e_j)=f(f_i,f_j)=0$ and $f(e_i,f_j)=\delta_{i,j}$ for all $i,j$, and in the unitary case,  $f(x,x)=1$ and $f(e_i,x)=f(f_i,x)=0$ for all $i$. 
Note that $e_1+e_2+f_2$ is a singular vector in either case. Now $\langle e_2\rangle$, $\langle f_2\rangle$ and $\langle e_1+e_2+f_2\rangle $ are three pairwise non-adjacent neighbours of $\langle e_1\rangle$, so
they have a common neighbour $w\in \Gamma_2(\langle e_1\rangle)$  by Lemma~\ref{lemma:sdividest}. But $w=\langle y\rangle$ for some non-zero singular vector $y$, and $0=f(y,e_2)=f(y,f_2)=f(y,e_1+e_2+f_2)$, so $f(y,e_1)=0$, a contradiction.

Thus $\mathcal{Q}$ is  $Q_4(3)$ or $Q_5^-(s)$ for $s\in \{2,3,4\}$. In order to complete the proof of Theorem~\ref{thm:girth3quad}, we must  show that the point graph of $Q_5^-(2)$ is $5$-CH but not $6$-CH,  and also that  the  point graphs of $Q_4(3)$ and $Q_5^-(s)$ for $s\in \{3,4\}$ are $4$-CH but not $5$-CH. The point graph of  $Q_5^-(2)$ is not $6$-CH by Lemma~\ref{lemma:sdividest} since $Q_5^-(2)$
has order $(2,4)$, and it is routine to verify that it  is $5$-CH. Similarly, the point graph of  $Q_4(3)$ is not $5$-CH by Lemma~\ref{lemma:sdividest} since $Q_4(3)$ has order $(3,3)$, and it is routine to verify that it  is $4$-CH. It is also routine to verify that the point graph of $Q_5^-(s)$ is $4$-CH for $s\in \{3,4\}$, so it remains to show that this graph is not $5$-CH. 

Let $s\in \{3,4\}$. Let $(V,Q)$ be the quadratic space corresponding to $Q_5^-(s)$, and let $f$ be the bilinear form  associated with $Q$. By \cite[Proposition 2.5.3]{KleLie1990}, $V$ has a basis $\{e_1,e_2,f_1,f_2,x,y\}$ where  $Q(e_i)=Q(f_i)=f(e_i,x)=f(f_i,x)=f(e_i,y)=f(f_i,y)=f(e_i,e_j)=f(f_i,f_j)=0$ and $f(e_i,f_j)=\delta_{i,j}$ for all $i,j$, and $Q(x)=1$, $f(x,y)=1$ and $Q(y)=\alpha$ for some $\alpha\in\mathbb{F}_s$ such that the polynomial $X^2+X+\alpha$ is irreducible over $\mathbb{F}_s$. If $s=3$, then $\alpha=-1$, and if $s=4$, then $\alpha\in \mathbb{F}_4\setminus \{0,1\}$.  Now $\langle e_2\rangle$, $\langle f_2\rangle$, $\langle e_2-f_2+x\rangle$ and $\langle e_1+e_2-\alpha^2f_2+\alpha x\rangle $ are four pairwise non-adjacent neighbours of $\langle e_1\rangle$ in the point graph of $Q_5^-(s)$, so if this graph is $5$-CH, then  these four vertices  have a common neighbour $w$ whose distance from $\langle e_1\rangle$ is $2$ by Lemma~\ref{lemma:sdividest}. But $w=\langle z\rangle$ for some non-zero singular vector $z$, and 
 $0=f(z,e_2)=f(z,f_2)=f(z,e_2-f_2+x)=f(z,e_1+e_2-\alpha^2f_2+\alpha x)$, so $f(z,e_1)=0$, a contradiction.
\end{proof}
 
\section{Graphs with girth $4$}
\label{s:girth4}

In this section, we prove Theorem \ref{thm:girth4}. Note that any graph $\Gamma$  with girth $4$ and valency $n$ is locally $n\cdot K_1$, and if the parameter $c_2$ is defined for $\Gamma$, then $c_2>1$.  First we have some preliminary observations.

\begin{lemma}
\label{lemma:g4}
Let $\Gamma$ be a locally finite connected $5$-CH graph with girth $4$. \begin{itemize}
\item[(i)]  If $\diam(\Gamma)\geq 4$, then $c_2(\Gamma)=c_3(\Gamma)$.
\item[(ii)] If $\diam(\Gamma)\geq 3$ and $\Gamma$ has valency at least $4$, then $c_2(\Gamma)\geq 3$.
\end{itemize}
\end{lemma}

\begin{proof}
Let $G:=\Aut(\Gamma)$. Since $\Gamma$ is $5$-CH, the parameters $c_2$ and $c_3$ are defined. 

Suppose that $\diam(\Gamma)\geq 4$.  There exists a geodesic $(u_0,\ldots,u_{4})$ in $\Gamma$. Note that $c_2\leq c_3$ since $\Gamma(u_{1})\cap \Gamma(u_3)\subseteq \Gamma(u_3)\cap \Gamma_{2}(u_0)$.  If $c_2<c_3$, then there exists $v\in \Gamma_{2}(u_0)$ such that $v$ is adjacent to $u_3$ but not $u_{1}$, so there exists $g\in G$ such that $u_{4}^g=v$ and $u_j^g=u_j$ for $0\leq j\leq 3$, a contradiction. Thus (i) holds.

Suppose that $\diam(\Gamma)\geq 3$ and $\Gamma$ has valency at least $4$, but $c_2=2$.
 There exists a geodesic $(u_0,u_1,u_2,u_3)$ in $\Gamma$. Since $d_\Gamma(u_1,u_3)=2$ and $c_2=2$, there exists a unique $v\in \Gamma_2(u_0)\setminus\{u_2\} $ such that $v$ is adjacent to $u_1$ and $u_3$. Similarly, there exists a unique $w\in \Gamma(u_0)\setminus \{u_1\}$ such that $w$ is adjacent to $u_0$ and $v$, and there exists a unique $x\in \Gamma(u_0)\setminus \{u_1\}$ such that $x$ is adjacent to $u_0$ and $u_2$. Note that $x\neq w$, or else $u_2$ and $v$ are vertices at distance two with common neighbours $u_1$, $w$ and $u_3$, a contradiction. Since $\Gamma$ has valency at least $4$, there exists $y\in \Gamma(u_0)\setminus \{u_1,x,w\}$. Since $\Gamma$ is $5$-CH, there exists $g\in \Aut(\Gamma)$ such that $u_i^g=u_i$ for $0\leq i\leq 3$ and $w^g=y$. But then $v^g=v$, so $w^g=w$, a contradiction. Thus (ii) holds.
\end{proof}

\begin{thm}
\label{thm:girth4plus}
Let $\Gamma$ be a finite connected $5$-CH graph with girth $4$ and valency $n$. Then $\Gamma$ is one of the following.
\begin{itemize}
\item[(i)] $K_{n,n}$ where $n\geq 2$.
\item[(ii)]  $K_2\times K_{n+1}$  where $n\geq 3$.
\item[(iii)] The folded $5$-cube $\Box_5$. 
\end{itemize}
In particular, $\Gamma$ is CH.
\end{thm}

\begin{proof}
Let  $G:=\Aut(\Gamma)$. Let $u\in V\Gamma$ and $H:=G_u^{\Gamma(u)}$. Since $\Gamma$ has girth $4$ and valency $n$, it is locally $n\cdot K_1$, and $c_2\geq 2$. If $\Gamma$ is one of the graphs listed in (i), (ii) or (iii), then $\Gamma$ is CH by Theorem~\ref{thm:CH}, and the proof is complete, so we assume otherwise. Let $r:=\min(n,4)$. Now $H$ is  an $r$-transitive permutation group of degree $n$, so by Theorem \ref{thm:4trans}, either $A_n\unlhd H$, or $H\simeq\M_n$ where $n\in \{11,12,23,24\}$.

First suppose that $A_n\unlhd H$. By~\cite[Theorem 4.5]{Cam1975},  $\Gamma$ is one of the following:  the $n$-cube $Q_n$ where $n\geq 4$ (since $Q_2\simeq K_{2,2}$ and $Q_3\simeq K_2\times K_4$), the folded $n$-cube $\Box_n$ where $n\geq 6$ (since $\Box_4\simeq K_{4,4}$), the incidence graph of the unique $2$-$(7,4,2)$ design, or the incidence graph of the unique $2$-$(11,5,2)$ design. In each case, $\diam(\Gamma)\geq 3$ and $n\geq 4$, but $c_2=2$, contradicting Lemma~\ref{lemma:g4}(ii). 

Thus $H\simeq \M_n$ where $n\in \{11,12,23,24\}$. In particular, $\Gamma$ does not have valency $2^m$ or $(m+1)(m^2+5m+5)$ for any positive integer $m$,  so~\cite[Theorem 4.4]{Cam1975} implies that either $\Gamma$ is a Hadamard graph, or $c_2=2$. If $\Gamma$ is a Hadamard graph, then it is distance-transitive with intersection array $\{2a,2a-1,a,1;1,a,2a-1,2a\}$ for some positive integer $a$ (see~\cite[\S 1.8]{BroCohNeu1989}),  but then $c_2\neq c_3$, contradicting Lemma~\ref{lemma:g4}(i). Thus $c_2=2$, so $\diam(\Gamma)=2$ by Lemma~\ref{lemma:g4}(ii). Now the polynomial $X^2+2X+(2-n)$ has integer roots by~\cite[Theorem 1.3.1]{BroCohNeu1989}, a contradiction. 
 \end{proof}
 
 \begin{proof}[Proof of Theorem $\ref{thm:girth4}$]
 By Remark~\ref{remark:connected}, we may assume that $\Gamma$ is connected. 
Then Theorem~\ref{thm:girth4plus} implies that (i) holds. Using Lemma~\ref{lemma:detkCH}, it is routine to verify that the $n$-cube is $4$-CH for $n\geq 4$, as is the folded $n$-cube  for $n\geq 6$, so (ii) holds.
 \end{proof}

 \section{Graphs with girth at least $5$}
\label{s:girth5}
 
In this section, we prove several results about the structure of finite $k$-CH graphs with girth at least $5$, and we then use these results to prove  Theorem \ref{thm:girth5}. We begin with a result about $7$-arc-transitive graphs, which requires the following standard definition.

Let $G$ be a group with a subgroup $H$ and element $a$ such that $a^2\in H$. Define $\Gamma(G,H,a)$ to be the graph whose vertices are the right cosets of $H$ in $G$, where two cosets $Hx$ and $Hy$ are adjacent whenever $xy^{-1}\in HaH$. The action of $G$ on $V\Gamma(G,H,a)$ by right multiplication induces an arc-transitive group of automorphisms. Conversely, if $\Gamma$ is a $G$-arc-transitive  graph with no isolated vertices, then for an arc $(u,v)$, there exists $g\in G$ such that $u^g=v$ and $v^g=u$, and  $\Gamma\simeq \Gamma(G,G_u,g)$ (see~\cite[Proposition~A.3.1]{BroCohNeu1989}).

Weiss~\cite{Wei1987} proved that there exists a finitely presented infinite group $R_{4,7}$  such that the automorphism group $G$ of any finite connected quartic $7$-arc-transitive graph   is a homomorphic image of $R_{4,7}$. Conder and Walker~\cite{ConWal1998} defined an equivalent presentation for $R_{4,7}$ in order to prove that there are infinitely many  finite connected quartic $7$-arc-transitive graphs. We use their presentation of $R_{4,7}$ to prove the following. 

\begin{prop}
\label{prop:7trans}
Let $\Gamma$ be a finite connected $7$-arc-transitive graph with valency $4$. Then $\Gamma$ is $6$-CH but not $7$-CH.
\end{prop}

\begin{proof}
Let $G:=\Aut(\Gamma)$. By~\cite{Wei1987,ConWal1998}, $G=\langle h,p,q,r,s,t,u,v,b\rangle$ where $h$ has order $4$, the elements $p,q,r,s,t,u$ and $hu$ have order $3$, the elements $v,b,uv$ and $huv$ are involutions, and the relations given for the definition of $R_{4,7}$ on~\cite[p.\ 622]{ConWal1998} are satisfied; we will use these relations throughout this proof without reference.  Further, $\Gamma\simeq \Gamma(G,H,b)$ where $H:=\langle h,p,q,r,s,t,u,v\rangle$. We may assume  that $\Gamma=\Gamma(G,H,b)$.  By~\cite[\S 2]{ConWal1998}, $H=\langle p,q,r,s,t\rangle{:}\langle h,u,v\rangle$  where $\langle p,q,r,s,t\rangle$ has order $3^5$ and $\langle h,u,v\rangle\simeq \GL_2(3)$, and the group $K:=\langle h^2,p,q,r,s,t,u,v\rangle=H\cap b^{-1}Hb$ is the stabiliser of the arc $(H,Hb)$ and has index $4$ in $H$. Thus the set of cosets of $K$ in $H$ is 
$$\{K,Kh^{-1}=Kh,Khv=Khu,Khu^{-1}\},$$ 
and for $x\in G$ and any transversal $T$ of $K$ in $H$, the neighbourhood $\Gamma(Hx)=\{Hbyx : y\in T\}$.   

 By Lemma~\ref{lemma:girth}, $\Gamma$ has girth at least $12$, so every connected  induced subgraph of $\Gamma$ of order at most $11$ is a tree. In particular, if $\Sigma$ is a connected induced subgraph of $\Gamma$ of order at most $10$, then any vertex in $V\Gamma\setminus V\Sigma$ is adjacent to at most one vertex in $V\Sigma$.

First we show that $\Gamma$ is not $7$-CH. Consider the induced subgraph $\Delta$ of $\Gamma$ in Figure~\ref{fig:not7CH}. (Note that $Hbhbh^{-1}$ is adjacent to $Hbh$ since $Hbh=Hbh^{-1}$.)
\begin{figure}
\includegraphics[page=3,height=2.8cm]{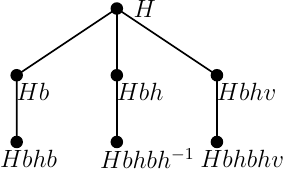}
\caption{An obstruction  for $7$-connected-homogeneity}
\label{fig:not7CH}
\end{figure}
Let $S$ be the pointwise stabiliser in $G$ of $V\Delta\setminus \{Hbhbhv\}$. To show that $\Gamma$ is not $7$-CH, it suffices to show that $S$ is not transitive on $\Gamma(Hbhv)\setminus \{H\}$. Let $T$ be the pointwise stabiliser of  $V\Delta\setminus \{Hbhv,Hbhbhv\}$. By~\cite[\S 2]{ConWal1998}, $T=\langle p,q,r,t^{-1}h^2,s^{-1}t^{-1}uv\rangle$ and $|T|=108$. Now $s^{-1}t^{-1}uv$ maps  $Hbhu^{-1}$ to $Hbhv$, so $T$ is transitive on $\Gamma(H)\setminus \{Hb,Hbh\}=\{Hbhv,Hbhu^{-1}\}$, while $S=T_{Hbhv}$, so $[T:S]=2$. Since $p$, $q$ and $r$ have odd order, they must be elements of $S$. It is routine to verify that the involution $t^{-1}h^2\in S$ and $\langle p,q,r,t^{-1}h^2\rangle\simeq C_3^3{:}C_2$, so $S=\langle p,q,r,t^{-1}h^2\rangle$. Now $S$ fixes the vertex $Hbhbhv\in \Gamma(Hbhv)\setminus \{H\}$, so $\Gamma$ is not $7$-CH.

Next we prove that $\Gamma$ is $6$-CH. Since $\Gamma$ is $7$-arc-transitive, we only need to consider those connected induced subgraphs of $\Gamma$ of order at most $6$ that are not path graphs; these are described in Figure~\ref{fig:6CH}.   
For  each  induced subgraph $\Delta$ of $\Gamma$ in Figure~\ref{fig:6CH},
\begin{figure}
\includegraphics[page=2,height=10cm]{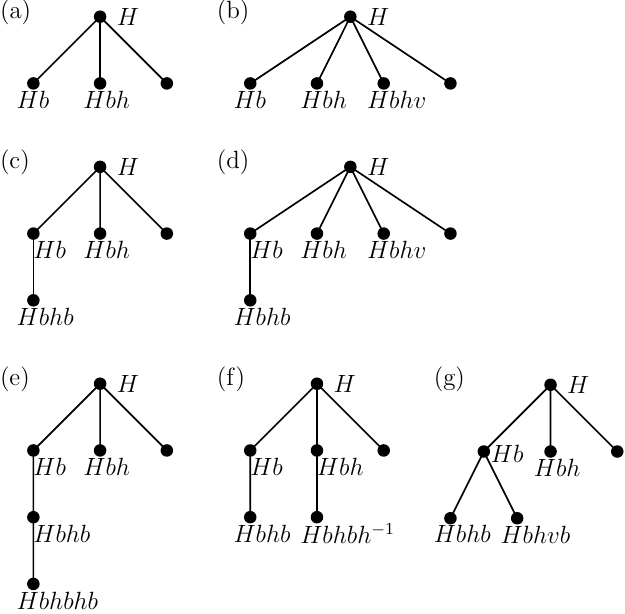}
\caption{The graphs for $6$-connected-homogeneity}
\label{fig:6CH}
\end{figure}
the vertices of some connected induced subgraph $\Sigma$  of order $|V\Delta|-1$ have been labelled with elements of $V\Gamma$. To show that $\Gamma$ is $6$-CH, by Lemma~\ref{lemma:detkCH}, it suffices to verify the following: for  each   $\Delta$  in Figure~\ref{fig:6CH}, the pointwise stabiliser $P$ in $G$ of $V\Sigma$ is transitive on $\Gamma(H)\setminus V\Sigma$. Cases (b) and (d) are trivial since $|\Gamma(H)\setminus V\Sigma|=1$. In all remaining cases, $\Gamma(H)\setminus V\Sigma=\{Hbhv,Hbhu^{-1}\}$. For cases (a), (c), (e) and (f),  the element $r^{-1}s^{-1}t^{-1}uv$ lies in $P$ and maps $Hbhu^{-1}$ to $Hbhv$. In case (g), the element $(t^{-1}h^2)(t^{-1}uv)$ lies in $P$ and maps $Hbhu^{-1}$ to $Hbhv$.
\end{proof}

\begin{remark}
If $\Gamma$  is the incidence graph of the split Cayley hexagon of order $(3,3)$---i.e., the generalised hexagon associated with the  group $G_2(3)$---then $\Gamma$ is $7$-arc-transitive with valency~$4$~\cite{Wei1987}, so $\Gamma$ is $6$-CH but not $7$-CH by Proposition~\ref{prop:7trans}.
\end{remark}

In order to prove Theorem~\ref{thm:girth5}, we first establish two more detailed results: Theorem~\ref{thm:girth5diam2} concerns those graphs with  diameter $2$, while Theorem~\ref{thm:girth5strong} concerns those graphs with diameter at least $3$. 
Observe that if $\Gamma$ is a finite connected graph with valency $2$, then $\Gamma\simeq C_n$ for some $n$, so in what follows, we focus on the case where $\Gamma$ has valency at least $3$.

First we consider the case where $\Gamma$ has diameter $2$. The Hoffman-Singleton graph is  a strongly regular graph with parameters $(50,7,0,1)$ and girth $5$; it has automorphism group $\PSigmaU_3(5)$ and point stabiliser $S_7$. See \cite[\S 13.1]{BroCohNeu1989} for several constructions of this graph. 

\begin{thm}
\label{thm:girth5diam2}
Let $\Gamma$ be a finite connected graph with girth at least $5$, valency at least $3$, and diameter $2$. If $\Gamma$ is $3$-CH, then  one of the following holds.
\begin{itemize}
\item[(i)] $\Gamma$ is   the Petersen graph. Here $\Gamma$ is CH.
\item[(ii)] $\Gamma$ is the Hoffman-Singleton graph. Here $\Gamma$ is $5$-CH but not $6$-CH.
\end{itemize}
\end{thm}

\begin{proof}
Since $\Gamma$ has girth at least $5$, it is locally $(t+1)\cdot K_1$ for some $t\geq 2$, and $c_2(\Gamma)=1$.  Since $\Gamma$ has diameter $2$,  it follows that $a_2(\Gamma)=t$. In particular, $\Gamma$ has girth $5$, so $\Gamma$ is a Moore graph (see \cite[\S 6.7]{BroCohNeu1989}). Since $\Gamma$ is $3$-CH, it is distance-transitive, so  either $t=2$ and $\Gamma$ is the Petersen graph, in which case (i) holds by Theorem~\ref{thm:CH}, or $t=6$ and $\Gamma$ is the Hoffman-Singleton graph~\cite{Asc1971,HofSin1960}.

Suppose that $\Gamma$ is the Hoffman-Singleton graph. Let $u\in V\Gamma$, write $\Gamma(u)=\{v,x,x_1\ldots,x_5\}$, and let $w_1,w_2\in \Gamma_2(u)\cap\Gamma(v)$. Let $\Delta_i$ be the subgraph of $\Gamma$ induced by $\{u,v,x,w_1,w_2,x_i\}$ for $i\in \{1,\ldots,5\}$. For each $i$, the graph $\Delta_i$ is isomorphic to the tree on $6$ vertices with two vertices of valency $3$. However, using {\sc  Magma}~\cite{Magma}, we determine that there exists $i\in \{1,\ldots,5\}$ such that the pointwise stabiliser in $\Aut(\Gamma)$ of $\{u,v,x,w_1,w_2\}$ also fixes $x_i$, so $\Gamma$ is not $6$-CH. Using~\cite{Magma,web-atlas}, it is routine to verify that $\Gamma$ is $5$-CH (see Remark~\ref{remark:magma}).
 \end{proof}

Next we consider the case where $\Gamma$ has diameter at least $3$. For $n\geq 4$, the odd graph $O_n$  (see~\cite[\S 9.1D]{BroCohNeu1989}) has girth $6$, valency $n$ and diameter $n-1$, and this graph is  $4$-CH but  not $5$-CH, as it is not $4$-arc-transitive. On the other hand,  we will see shortly that the valency of a $5$-CH graph with girth at least $5$ is very restricted. We will also see that the only  $5$-CH graphs with girth at least $5$ and diameter $3$ are the incidence graphs of the projective planes $\PG_2(q)$ for $2\leq q\leq 4$. Note that   the incidence graph of the Fano plane $\PG_2(2)$ is often called the Heawood graph. Recall the definition of an $s$-transitive graph from \S\ref{s:defn}.

\begin{thm}
\label{thm:girth5strong}
Let $\Gamma$ be a finite connected $5$-CH graph with girth at least $5$, valency $n\geq 3$, and $\diam(\Gamma)\geq 3$.  Then $3\leq n\leq 5$,  $\Gamma$ is  $4$-arc-transitive, and  one of the following holds.
\begin{itemize}
\item[(i)] $\diam(\Gamma)=3$ and $\Gamma$ is the incidence graph of the projective plane $\PG_2(q)$ for $2\leq q\leq 4$. Here $\Gamma$ is $6$-CH but not $7$-CH.
\item[(ii)] $\diam(\Gamma)=4$ and $\Gamma$ is the incidence graph of the generalised quadrangle $W_3(q)$ for $q=2$ or $4$. Here $\Gamma$ is $6$-CH but not $7$-CH.
\item[(iii)] $\diam(\Gamma)\geq 5$ and $\Gamma$ is not $7$-CH. Further, if $\Gamma$ is $6$-CH, then  either $\Gamma$ is $5$-transitive and $n=3$ or $5$, or $\Gamma$ is $7$-transitive and $n=4$.
\end{itemize}
\end{thm}

\begin{proof}
  Let $G:=\Aut(\Gamma)$, let $u\in V\Gamma$ and let $H:=G_u^{\Gamma(u)}$. The graph $\Gamma$ is locally $(t+1)\cdot K_1$ for some $t\geq 2$. Now  $H$ is an $r$-transitive permutation group of degree $t+1$   where $r:=\min{(t+1,4)}$, so by Theorem~\ref{thm:4trans},  either $A_{t+1}\unlhd H\leq S_{t+1}$, or  $H$ is the simple group $ \M_{t+1}$ where $t+1\in \{11,12,23,24\}$. Let $(u_0,\ldots,u_4)$ be a $4$-arc. Since $\Gamma$ has diameter at least $3$, it  has girth at least $6$ by Lemma~\ref{lemma:c2=1}(i), so  the subgraph induced by $\{u_0,\ldots,u_4\}$ is a path graph with $5$ vertices. Thus $\Gamma$ is $4$-arc-transitive. In particular, $\Gamma$ is $s$-transitive for some $s\geq 4$, so by~\cite{Wei1981},  $\PSL_2(t)\unlhd H$ and  $t$ is a power of a prime $\ell$ such that either
$s=4$, or $s=2\ell+1$ and $\ell\leq 3$. If $t\geq 5$, then $H$ has a unique non-abelian composition factor, namely $A_{t+1}$ or $M_{t+1}$, but neither of these groups is isomorphic to the simple group $\PSL_2(t)$, a contradiction.  Thus  $t\in\{2,3,4\}$ and $n\in \{3,4,5\}$.

If $\diam(\Gamma)=3$, then $\Gamma$ is  distance-transitive with intersection array $\{t+1,t,t; 1,1,c_3\}$ where $c_3=1$ or $t+1$ by  Lemma~\ref{lemma:c2=1}. We claim that $\Gamma$ is the incidence graph of the projective plane  $\PG_2(t)$. If $t=2$, then $\Gamma$ is the Heawood graph by~\cite[Theorem~7.5.1]{BroCohNeu1989}, so the claim holds. If $t=3$ or $4$, then $\Gamma$ is the point graph of a generalised hexagon of order $(1,t)$ by~\cite[Theorem~7.5.3]{BroCohNeu1989}, so $\Gamma$ is the incidence graph of a projective  plane of order $t$ (see \cite[\S6.5]{BroCohNeu1989}); since  $\PG_2(t)$ is the unique projective plane of order $t$ when $t=3$ or $4$, the claim holds. Next we show that $\Gamma$ is not $7$-CH.  
 Choose a point $p$ of $\PG_2(t)$, and let $\ell_1$, $\ell_2$ and $\ell_3$ be pairwise distinct lines on $p$. Let $q_1$ and $q_2$ be points on $\ell_1\setminus\{p\}$ and $\ell_2\setminus\{p\}$ respectively. Let $\ell$ be the unique line on $q_1$ and $q_2$, and let $q_3$ be the unique point on $\ell$ and $\ell_3$. For $x\in \ell_3\setminus \{p\}$, let $\Delta_x$ be the subgraph of $\Gamma$ induced by $\{p,\ell_1,\ell_2,\ell_3,q_1,q_2,x\}$, and observe that $\Delta_x\simeq \Delta_y$ for all $y\in \ell_3\setminus \{p\}$. Observe also that $q_3\in \ell_3\setminus \{p\}$ and $|\ell_3\setminus \{p\}|=t\geq 2$. 
 Suppose that $g\in \Aut(\Gamma)$ fixes $\{p,\ell_1,\ell_2,\ell_3,q_1,q_2\}$ pointwise. Note that $g$ is a collineation of $\Gamma$; that is, $g$ maps points to points and lines to lines. Since $q_1$ and $q_2$ are distinct points on $\ell$, we must have $\ell^g=\ell$. Now $q_3^g$ lies on $\ell^g=\ell$ and $\ell_3^g=\ell_3$, so $q_3^g=q_3$.  Thus $\Gamma$ is not $7$-CH. It is routine to verify that $\Gamma$ is $6$-CH, so (i) holds.

 If $\diam(\Gamma)=4$, then $\Gamma$ is distance-transitive with intersection array $\{t+1,t,t,t;1,1,1,c_4 \}$ by Lemma~\ref{lemma:c2=1}, so $\Gamma$ is the incidence graph of the generalised quadrangle $W_3(t)$ where $t\in \{2,4\}$ by~\cite[Theorems 7.5.1 and 7.5.3]{BroCohNeu1989}. We may view the points and lines of $W_3(t)$ as points and lines respectively  of $\PG_2(t)$, and the above proof shows that $\Gamma$ is not $7$-CH (even though the line $\ell$ is not a line of $W_3(t)$). It is routine to verify that $\Gamma$ is $6$-CH, so (ii) holds.

We may therefore assume that $\diam(\Gamma)\geq 5$. Now $a_i=0$ and $c_i=1$ for $1\leq i\leq 3$ by  Lemma~\ref{lemma:c2=1}, so $\Gamma$ has girth at least $8$. Thus the subgraph induced by a $5$-arc (or $6$-arc) is a path graph with $6$ (or $7$)  vertices. 
If $\Gamma$ is $6$-CH, it follows that $\Gamma$ is $5$-arc-transitive, and since $s=2\ell+1$ where $2\leq \ell\leq 3$ and $t$ is a power of $\ell$, we conclude that either $\Gamma$ is $5$-transitive and $n=3$ or $5$, or $\Gamma$ is $7$-transitive and $n=4$. Similarly, if $\Gamma$ is $7$-CH, then $\Gamma$ is $6$-arc-transitive and therefore $7$-arc-transitive with valency $4$, but no such graph exists by Proposition~\ref{prop:7trans}. Thus (iii) holds.
\end{proof}

\begin{proof}[Proof of Theorem $\ref{thm:girth5}$]
 (i) Let $\Gamma$ be a finite $7$-CH graph with girth at least $5$.  By Remark~\ref{remark:connected}, we may assume that $\Gamma$ is connected. If $\Gamma$ has valency $0$ or $1$, then $\Gamma$ is CH. If $\Gamma$ has valency $2$, then $\Gamma\simeq C_n$ for some $n\geq 5$, so $\Gamma$ is CH. If $\Gamma$ has valency at least $3$, then by Theorems~\ref{thm:girth5diam2} and~\ref{thm:girth5strong}, $\Gamma$ is the Petersen graph and $\Gamma$ is CH.  

(ii) By~\cite{ConWal1998}, there are infinitely many finite connected quartic $7$-arc-transitive graphs, all of which have girth at least $12$ by Lemma~\ref{lemma:girth}, and by Proposition~\ref{prop:7trans}, any such graph is $6$-CH but not $7$-CH.

(iii) Let $\Gamma$ be a finite graph with valency $4$ and girth at least $7$. Note that $\Gamma$ is $7$-arc-transitive if and only if $\Gamma$ is a disjoint union of connected $7$-arc-transitive graphs, all of which are isomorphic, so by Remark~\ref{remark:connected}, we may assume that $\Gamma$ is connected. 
Note that the Petersen graph has girth~$5$, and the incidence graph of  $\PG_2(q)$ has girth $6$.  If $\Gamma$ is $6$-CH, then by  Theorems~\ref{thm:girth5diam2} and~\ref{thm:girth5strong}, $\Gamma$ is $7$-arc-transitive. Conversely, if $\Gamma$ is $7$-arc-transitive, then $\Gamma$ is $6$-CH by Proposition~\ref{prop:7trans}. 
\end{proof}

We finish this section with the following observation.

\begin{prop}
\label{prop:4trans}
Any finite cubic $4$-arc-transitive graph  is $5$-CH.
\end{prop}

\begin{proof}
Let $\Gamma$ be a finite $4$-arc-transitive graph with valency $3$.  By Lemma~\ref{lemma:girth}, $\Gamma$ has girth at least $6$,  so the only connected induced subgraphs of $\Gamma$ with order at most $5$ that are not $s$-arcs  are the  trees with order $4$ and $5$ that contain a vertex of valency $3$. Since $\Gamma$ has valency $3$, it follows from Lemma~\ref{lemma:detkCH} that $\Gamma$ is $5$-CH.
\end{proof}

\section{Locally disconnected graphs with girth $3$ and $c_2=1$}
\label{s:LDbad}

Recall that a finite $3$-CH graph $\Gamma$ is locally disconnected with girth $3$ if and only if $\Gamma$ is locally $(t+1)\cdot K_s$ for some integers $t\geq 1$ and $s\geq 2$ (see Lemma~\ref{lemma:3CH}). In this section, we consider such graphs $\Gamma$ for which $c_2=1$.  In \S\ref{ss:t=1}, we consider the case where $t=1$;  in particular, we prove Theorem~\ref{thm:2Ks}. In \S\ref{ss:general}, we briefly consider the case where $t\geq 2$.

\subsection{The case where $t=1$}
\label{ss:t=1}

\begin{lemma}
\label{lemma:linegraph}
Let $\Gamma$ be a finite connected graph. For  $s\geq 2$, the following are equivalent.
\begin{itemize}
\item[(i)] $\Gamma$ is locally $2\cdot K_s$ and $c_2(\Gamma)=1$.
\item[(ii)] $\Gamma$ is the line graph of a finite connected graph  with girth at least $5$ and valency $s+1$.
\end{itemize}
\end{lemma}

\begin{proof}
If (i) holds, then (ii) holds by~\cite[Proposition 1.2.1]{BroCohNeu1989}. Conversely, if (ii) holds, then it is routine to verify that (i) holds.
\end{proof}

Let $\Gamma_1$ and $\Gamma_2$ be  finite connected graphs where $E\Gamma_1\neq\varnothing$. If $\varphi:\Gamma_1\to\Gamma_2$ is an isomorphism, then there is a natural isomorphism  $\hat{\varphi}:L(\Gamma_1)\to L(\Gamma_2)$ defined by $\{u,v\}\mapsto \{u\varphi,v\varphi\}$ for all $\{u,v\}\in E\Gamma_1$. Conversely, if  $\Gamma_1$ and $\Gamma_2$   have at least five vertices, and if $\psi:L(\Gamma_1)\to L(\Gamma_2)$ is an isomorphism,  then there exists a unique isomorphism $\varphi:\Gamma_1\to \Gamma_2$ such that $\psi=\hat{\varphi}$ by~\cite[Theorem 8.3]{Har1969}. It is routine to verify that this also holds when $\Gamma_1$ and $\Gamma_2$ have at least three vertices and do not contain any triangles. In particular, if $\Gamma$ is a finite connected regular graph with girth at least $5$ and valency at least $3$, then $\Gamma$ contains a cycle with length at least $5$, so  there is a group isomorphism of $\Aut(\Gamma)$ onto $\Aut(L(\Gamma))$ defined by $g\mapsto \hat{g}$ for all $g\in \Aut(\Gamma)$.

\begin{lemma}
\label{lemma:linegraphCH}
Let $\Gamma$ be a finite connected regular graph with girth at least $5$ and valency at least~$3$. For $k\geq 2$, the following are equivalent.
\begin{itemize}
\item[(i)] The line graph $L(\Gamma)$ is $k$-CH.
\item[(ii)] $\Gamma$ is $(k+1)$-CH and has girth at least $k+2$.
\end{itemize}
\end{lemma}

\begin{proof}
Suppose that $L(\Gamma)$ is $k$-CH. The graph $\Gamma$ has girth at least $5$ and valency at least $3$, so   there exists an  induced subgraph $\Delta$ of $\Gamma$ such that $\Delta$ is  a path graph with $5$ vertices. Write $E\Delta=\{e_1,e_2,e_3,e_4\}$ where $e_i$ is incident with $e_{i+1}$ for $1\leq i\leq 3$. Now   $L(\Delta)$ is a path graph with $4$ vertices that is an induced subgraph of $L(\Gamma)$, so if $\diam(L(\Gamma))=2$, then there exists $e\in E\Gamma$ such that $e$ is incident with $e_1$ and $e_4$, but then $e\in E\Delta$, a contradiction. Thus $\diam(L(\Gamma))\geq 3$.

Let $r$ be the girth of $\Gamma$. Now $\Gamma$ has an induced subgraph $\Delta'$ such that $\Delta'\simeq C_r$, so $L(\Delta')$ is an induced subgraph of $L(\Gamma)$ that is  isomorphic to $C_r$.  
Let $\Gamma$ have valency $s+1$, and note that $s\geq 2$. By Lemma~\ref{lemma:linegraph}, $L(\Gamma)$ is locally $2\cdot K_s$ and $c_2(L(\Gamma))=1$.  Thus $r\geq k+2$ by Lemma~\ref{lemma:cyclebig}.

Now we prove that $\Gamma$ is $(k+1)$-CH. Let $\Delta_1$ and $\Delta_2$ be  connected induced subgraphs of $\Gamma$  where $3\leq |V\Delta_1|\leq k+1$, and let $\varphi:\Delta_1\to \Delta_2$ be an isomorphism. Since $\Gamma$ has girth at least $k+2$, the graph $\Delta_1$ is a tree and therefore has at most $k$ edges. Now $L(\Delta_1)$ and $L(\Delta_2)$ are connected induced subgraphs of $L(\Gamma)$  with order at most $k$, and $\hat{\varphi}:L(\Delta_1)\to L(\Delta_2)$ is an isomorphism, so there exists $g\in \Aut(\Gamma)$ such that $\hat{\varphi}$ extends to $\hat{g}\in \Aut(L(\Gamma))$, whence $\varphi$ extends to $g$, as desired.
In particular, we have proved that $\Gamma$ is $2$-arc-transitive; since $\Gamma$ has valency at least $2$, it follows that $\Gamma$ is  $2$-CH and therefore $(k+1)$-CH. Thus (ii) holds.

Conversely, suppose that $\Gamma$ is $(k+1)$-CH and has girth at least $k+2$. Let $\Delta_1$ and $\Delta_2$ be connected induced subgraphs of $L(\Gamma)$  with order at most $k$, and let $\psi:\Delta_1\to \Delta_2$ be an isomorphism. It is clear that $L(\Gamma)$ is vertex-transitive, so we may assume that $|\Delta_1|\geq 2$. Let $\Sigma_i$ be the subgraph of $\Gamma$ induced by the edges in $V\Delta_i$. Now $\Sigma_i$ is connected and $L(\Sigma_i)=\Delta_i$. In particular, there exists an isomorphism $\varphi: \Sigma_1\to \Sigma_2$ such that $\psi=\hat{\varphi}$. 
 Since $\Sigma_i$ has $|V\Delta_i|$ edges and $\Gamma$ has girth at least $k+2$, it follows that $\Sigma_i$ is a tree, so $|V\Sigma_i|=|V\Delta_i|+1$. If $\Sigma_i$ is not a (vertex) induced subgraph of $\Gamma$, then the graph induced by $V\Sigma_i$ contains a cycle of length at most $k+1$, a contradiction. Hence there exists $g\in \Aut(\Gamma)$ such that $\varphi$ extends to $g$, in which case $\psi=\hat{\varphi}$ extends to $\hat{g}\in \Aut(L(\Gamma))$. Thus (i) holds.
\end{proof}

\begin{remark}
By Lemma~\ref{lemma:linegraphCH}, any result from \S\ref{s:girth5} can be reinterpreted for locally $2\cdot K_s$ graphs where $s\geq 2$ and $c_2=1$. For example, if $\Gamma$ is the line graph of the incidence graph of $\PG_2(q)$ for $2\leq q\leq 4$, then $\Gamma$ is $4$-CH but not $5$-CH (since the incidence graph of $\PG_2(q)$ has girth $6$). Similarly, if $\Gamma$ is the line graph of the incidence graph of either $W_3(q)$ for $q=2$ or $4$, or the  split Cayley hexagon of order $(3,3)$, then $\Gamma$ is $5$-CH but not $6$-CH. 
\end{remark}

The following result is an easy consequence of Theorems \ref{thm:girth5diam2} and \ref{thm:girth5strong}.

\begin{thm}
\label{thm:2Ksplus}
Let $\Gamma$ be a finite connected graph  that is locally $2\cdot K_s$ where $s\geq 2$ and  $c_2(\Gamma)=1$. Then $\Gamma$ is not $6$-CH. If $\Gamma$ is $4$-CH, then $s\in \{2,3,4\}$ and $\Gamma$ is the line graph of a finite connected $4$-arc-transitive graph  with valency $s+1$.
\end{thm}

\begin{proof}
By Lemma~\ref{lemma:linegraph}, $\Gamma=L(\Sigma)$ for some finite connected graph $\Sigma$  with girth at least $5$ and valency $s+1$. If $\Gamma$ is $4$-CH, then $\Sigma$ is $5$-CH with girth at least $6$ by Lemma~\ref{lemma:linegraphCH}, so by Theorems~\ref{thm:girth5diam2} and~\ref{thm:girth5strong}, $\Sigma$ is $4$-arc-transitive and $s\in \{2,3,4\}$.   If $\Gamma$ is $6$-CH, then $\Sigma$ is $7$-CH and has girth at least $8$ by Lemma~\ref{lemma:linegraphCH}, but no such graph exists by Theorems~\ref{thm:girth5diam2} and~\ref{thm:girth5strong}.
\end{proof}

\begin{proof}[Proof of Theorem $\ref{thm:2Ks}$]
By Theorem~\ref{thm:2Ksplus} and Remark~\ref{remark:connected}, Theorem~\ref{thm:2Ks}(i) holds. By~\cite{ConWal1998}, there are infinitely many  finite connected $7$-arc-transitive graphs with valency $4$, all of which are $6$-CH by Proposition~\ref{prop:7trans}, so Theorem~\ref{thm:2Ks}(ii) follows from  Lemmas~\ref{lemma:girth}, \ref{lemma:linegraph} and~\ref{lemma:linegraphCH}.
\end{proof}

\begin{prop}
The line graph of a finite connected cubic $4$-arc-transitive graph  is $4$-CH.
\end{prop}

\begin{proof}
Let $\Gamma$ be a finite connected $4$-arc-transitive graph with valency $3$. Then $\Gamma$ is $5$-CH by Proposition~\ref{prop:4trans}, and  $\Gamma$ has girth at least $6$ by  Lemma~\ref{lemma:girth}, so $L(\Gamma)$ is $4$-CH by Lemma~\ref{lemma:linegraphCH}.
\end{proof}

\subsection{The  case where $t\geq 2$}
\label{ss:general}

For finite connected locally $(t+1)\cdot K_s$ graphs $\Gamma$ with $t\geq 2$, $s\geq 2$ and $c_2=1$, our only general results are Lemmas~\ref{lemma:c2=1} and~\ref{lemma:cyclebig}. Note that Kantor \cite{Kan1977} proved (without using the CFSG) 
that no finite connected  $3$-CH graph with girth $3$ and $c_2=1$ is strongly regular. 

In the following, we determine those $4$-CH  graphs $\Gamma$ that are  distance-transitive with valency at most $13$.  
The point graph of the Hall-Janko near octagon (see~\cite[\S 13.6]{BroCohNeu1989}) is locally $5\cdot K_2$ and distance-transitive with intersection array $\{10,8,8,2;1,1,4,5\}$. It has automorphism group $\J_2{:}2$, where $\J_2$ denotes  the Hall-Janko sporadic simple group.

\begin{prop}
\label{prop:badsmallval}
Let $\Gamma$ be a finite connected $4$-CH graph that is locally $(t+1)\cdot K_s$ where $t\geq 2$, $s\geq 2$ and $c_2(\Gamma)=1$. If $\Gamma$ is distance-transitive with valency at most $13$, then one of the following holds.
\begin{itemize}
\item[(i)] $\Gamma$ is the point graph of the   dual of the split Cayley hexagon of order $(2,2)$. Here $\Gamma$ is $4$-CH but not $5$-CH.
\item[(ii)] $\Gamma$ is the point graph of the  Hall-Janko near octagon. Here $\Gamma$ is $4$-CH but not $5$-CH.
\end{itemize}
\end{prop}

\begin{proof}
By~\cite{Kan1977}, $\Gamma$ is not strongly regular, so $d:=\diam(\Gamma)\geq 3$. By Lemma~\ref{lemma:c2=1}, $\Gamma$ has intersection array $\{s(t+1),st,st,b_3,\ldots,b_{d-1};1,1,c_3,\ldots,c_d\}$. In particular, $b_0=s(t+1)\geq 6$ and $b_1=b_2=st\geq 4$. Further, neither $b_0$ nor $b_1$ is prime, and $b_0-b_1=s\geq 2$. By~\cite[Theorem~7.5.3]{BroCohNeu1989}, one of the following holds.
\begin{itemize}
\item[(a)] $\Gamma$ is the point graph of a generalised hexagon of order $(2,2)$. Here  $\Gamma$ has intersection array $\{6,4,4;1,1,3\}$ and  is locally $3\cdot K_2$ with order $63$.
\item[(b)] $\Gamma$ is the point graph of the  Hall-Janko near octagon.  Here $\Gamma$ has intersection array $\{10,8,8,2;1,1,4,5\}$ and is locally $5\cdot K_2$ with order    $315$.
\item[(c)]   $\Gamma$ is the point graph of a generalised octagon of order $(2,4)$. Here $\Gamma$ has intersection array $\{10,8,8,8;1,1,1,5\}$ and is locally $5\cdot K_2$ with order $1755$. 
\item[(d)]  $\Gamma$ is the point graph of a generalised hexagon of order $(3,3)$. Here $\Gamma$ has intersection array $\{12,9,9;1,1,4\}$ and is locally $4\cdot K_3$ with  order $364$.
\item[(e)]  $\Gamma$ is the point graph of a generalised octagon of order $(4,2)$. Here $\Gamma$ has intersection array $\{12,8,8,8;1,1,1,3\}$ and $\Gamma$ is locally $3\cdot K_4$ with  order $2925$.
\end{itemize}
If (b) holds, then $\Gamma$ is not $5$-CH by Lemma~\ref{lemma:c2=1}, and using {\sc Magma}~\cite{Magma} and~\cite{web-atlas} 
 (see Remark~\ref{remark:magma}), it is routine to verify that $\Gamma$ is  $4$-CH, so (ii) holds. 

Thus we may assume that one of (a), (c), (d) or (e) holds. Now $\Gamma$ is the point graph of a distance-transitive generalised $n$-gon $\mathcal{S}$ for some $n$. By~\cite{BueVan1994}, one of the following holds: in (a), $\mathcal{S}$ is the split  Cayley hexagon of order $(2,2)$ or its dual; in (c), $\mathcal{S}$ is the Ree-Tits octagon of order $(2,4)$; in (d), $\mathcal{S}$ is the split Cayley hexagon of order $(3,3)$ (since this generalised hexagon is self-dual); and in (e),  $\mathcal{S}$ is the dual of the Ree-Tits octagon of order $(2,4)$.

If (c) or (d) holds, then   using~\cite{Magma,web-atlas}, it is routine to verify that $\Gamma$ is not $4$-CH (the tree of order $4$ with a vertex of valency $3$ fails). Similarly, if (e) holds, then using~\cite{Magma,web-atlas}, it is routine to verify that $\Gamma$ is not $3$-CH (the cycle of length~$3$ fails).

Lastly, suppose  that (a) holds. Neither graph is $5$-CH by Lemma~\ref{lemma:c2=1}. In order to differentiate between the split Cayley hexagon and its dual, here is a construction of the former:  its points are the one-dimensional totally singular subspaces of a quadratic space   on $V_7(2)$, and its $63$ lines are an orbit of $G_2(2)\leq \POmega_7(2)$  on the two-dimensional  totally singular subspaces. (See   \cite[\S2.4.13]{Van2012} for an explicit description of the lines.)
Using {\sc Magma}~\cite{Magma}, it is routine to verify that the point graph of the split Cayley hexagon is not $4$-CH  (the tree of order $4$ with a vertex of valency $3$ fails), while the point graph of the dual of the split Cayley hexagon is $4$-CH. Thus (i) holds.
\end{proof}

\bibliographystyle{acm}
\bibliography{jbf_references}

\end{document}